\documentclass[a4paper,12pt]{article}
\usepackage[english]{babel}
\usepackage[utf8]{inputenc}
\usepackage[T1]{fontenc}
\usepackage{tikz}
\usepackage{amsmath,amsfonts}
\usepackage{amssymb}
\usepackage{amsthm}
\usepackage[inner=20mm,outer=20mm,top=1.75cm,
bottom=3cm
]{geometry}

\usepackage{appendix}
\usepackage{graphicx,subfigure}
\usepackage{textcomp}
\usepackage{graphicx}

\usepackage{hyperref} 
\usepackage{tcolorbox} 
\usepackage{enumitem}
\newtheorem{theo}{Theorem}[section]

\newtheorem{lem}{Lemma}[section]
\newtheorem{prop}{Proposition}[section]
\newtheorem{corollary}{Corollary}[section]
\newtheorem{defi}{Definition}[section]
\newtheorem{assum}{Assumption}[section]
\newtheorem{remark}{Remark}[section]

\newcommand\eps{\varepsilon}

\newcommand\ka{\kappa}

\newcommand{\sgn}{\text{sgn}}

\def \p {{\partial}}

\def \induu {\mathbf{1}_{ \{u_{1}<0 \} }}
\def \indud {\mathbf{1}_{\{ u_{2}<0\}}}
\def \induz {\mathbf{1}_{\{u_{0}<0\} }}
\def \indvu {\mathbf{1}_{\{q_{1}<0\} }}
\def \indvd {\mathbf{1}_{\{q_{2}<0\}}}

\def \indw {\mathbf{1}_{\{w_{1}\geq 0 \}}}
\def \indww {\mathbf{1}_{\{w_{2}\geq 0 \}}}
\def \indwz {\mathbf{1}_{\{w_{0}\geq 0 \}}}
\def \indz {\mathbf{1}_{\{z_{1}\geq 0 \} }}
\def \indzz {\mathbf{1}_{\{z_{2}\geq 0 \}}}

\def \tu {\tilde{u}}
\def \tq {\tilde{q}}

\def \mT {\mathcal{T}}

\def \tU {\widetilde{U}}
\def \tW {\widetilde{W}}

\usepackage{authblk}
\begin{document}

\newcommand{\ue}{\frac{1}{\eps}}
\newcommand{\bphi}{{\boldsymbol{\phi} }}
\newcommand{\cS}{{\mathcal S}}
\newcommand{\BV}{{\rm BV}}
\newcommand{\bBV}{{\rm  {\mathbf B}{\mathbf V}}}

\newcommand{\e}{\eps}
\newcommand{\uzdu}{u^{0,\delta}_1}
\newcommand{\uzdd}{u^{0,\delta}_2}
\newcommand{\qzdu}{q^{0,\delta}_1}
\newcommand{\qzdd}{q^{0,\delta}_2}
\newcommand{\uzdz}{u^{0,\delta}_0}
\newcommand{\ubd}{u^\delta_b}
\newcommand{\chiu}[1]{\chi_{#1}}
\newcommand{\ti}[1]{\tilde{#1}}
\newcommand{\RR}{{\mathbb R}}
\newcommand{\nrm}[2]{\left\lVert#1\right\rVert_{#2}}
\newcommand{\dx}{{\partial_x}}
\newcommand{\dt}{{\partial_t}}

\title{Reduction of a model for sodium exchanges in kidney nephron}
\author[1,2,3]{Marta Marulli}
\author[2]{Vuk Milišić}
\author[2]{Nicolas Vauchelet}
\affil[1]{University of Bologna, Department of Mathematics, Piazza di Porta S. Donato 5, 40126 Bologna, Italy}
\affil[2]{Université Sorbonne Paris Nord, Laboratoire Analyse, Géométrie et Applications, LAGA, CNRS UMR 7539, F-93430, Villetaneuse, France}
\affil[3]{{Math\'ematiques Appliqu\'ees à Paris 5}, MAP5 CNRS UMR 8145,  Université Paris Descartes}

\date{}                     
\setcounter{Maxaffil}{0}
\renewcommand\Affilfont{\itshape\small}
\maketitle

\begin{abstract}

  This work deals with a mathematical analysis of sodium's transport in a  tubular architecture of a kidney nephron.
 The nephron is modelled by two counter-current tubules. Ionic exchange occurs at the interface between the tubules and the epithelium and between the epithelium and the surrounding environment (interstitium).
  From a mathematical point of view, this model consists of a 5$\times$5 semi-linear hyperbolic system.
  In the literature similar models neglect the epithelial layers.
  In this paper, we show rigorously that such models may be obtained by assuming that the permeabilities between  lumen and  epithelium are large. Indeed we show that when these grow,  solutions of the 5$\times$5
  system converge in a certain way to  solutions of a reduced 3$\times$3 system where no epithelial layer
  is present. The problem is defined on a bounded spacial domain with initial and boundary data. 
  Establishing $\BV$ compactness forces to introduce initial layers and to handle carefully 
 	the presence of lateral boundaries.
\end{abstract}


\medskip



\section{Introduction}

\section{Introduction}	
In this study, we consider a mathematical model for a particular component of the nephron, the functional unit of kidney. It describes the ionic exchanges through the nephron tubules in the Henle's loop.
The main function of the kidneys is the filtration of blood.
Through filtration, secretion and excretion of filtered metabolic wastes and toxins, the kidneys are able to maintain a certain homeostatic balance within cells. Despite the development of sophisticated models about water and electrolyte transport in the kidney, some aspects of the fundamental functions of this organ remain yet to be fully explained,\cite{anitalayton}.
For example, how a concentrated urine can be produced by the mammalian kidney when the animal is deprived of water remains not entirely clear. \\
The loop of Henle and its architecture play an important role in the concentrated or diluted urine formation.
In order to explain how an animal or a human being can produce a concentrated urine and what this mechanism depends on, we need to analyse the counter-current transport in the 'ascending' and 'descending' tubules. There the ionic exchanges between the cell membrane
and the environment where tubules are immersed, take place. \\
We consider a simplified model for sodium exchange in the kidney nephron. In this simplified version, the nephron is modelled by two tubules, one ascending and one descending, of length denoted $L$. Ionic exchanges and transport occur at the interface between the lumen and the epithelial layer (cell membrane) and at the interface between the cells and the interstitium (this term indicates all the space/environment that surrounds the tubules and blood vessels). A schematic representation for the model is given in Figure \ref{fig:1}.
If we denote $t\geq 0$ and $x\in (0,L)$ the time and space variables, respectively, the dynamics of ionic concentrations is modelled by the following semi-linear hyperbolic system (see e.g.\cite{MAMV,magali,tesp,magali2})

\begin{equation}\label{problem}
\begin{cases}
\p_{t}u_1 + \alpha\p_{x} u_1=J_1=2\pi r_1 P_1(q_{1}-u_1) \\
\p_{t}u_2 - \alpha\p_{x} u_2=J_2=2\pi r_2 P_2(q_{2}-u_2) \\
\p_{t}q_{1}=J_{1,e}=2\pi r_1 P_1(u_{1}-q_{1}) + 2\pi r_{1,e} P_{1,e}(u_{0}-q_{1}) \\
\p_{t}q_{2}=J_{2,e}=2\pi r_2 P_2(u_{2}-q_{2}) + 2\pi r_{2,e} P_{2,e}(u_{0}-q_{2})- G(q_{2})\\
\vspace{2mm} \p_{t}u_0=J_0=2\pi r_{1,e} P_{1,e}(q_1-u_0) + 2\pi r_{2,e} P_{2,e}(q_2-u_0) + G(q_2),
\end{cases}
\end{equation}
complemented with the boundary and initial conditions
\begin{equation}\label{problembord}
\begin{cases}
u_1(t,0)=u_b(t); \quad u_1(t,L)=u_2(t,L) \quad t > 0\\
u_1(0,x)=u_1^{0}(x); \quad u_2(0,x)=u_2^{0}(x); \quad u_{0}(0,x)=u_{0}^{0}(x); \\
q_1(0,x)=q_1^{0}(x); \quad q_2(0,x)=q_2^{0}(x).
\end{cases}
\end{equation}
In this model, we have used the following notations~:
\begin{itemize}
	\item $r_{i}:$ denote the radius for the lumen $i$ $([m])$.
	\item $r_{i,e}:$ denote the radius for the tubule $i$ with epithelium layer.
	\item Sodium's concentrations ($[mol/m^{3}$]) : \\
	$u_{i}(t,x):$  in the lumen $i$, \\
	$q_{i}(t,x):$  in the epithelium 'near' lumen, $i$ \\
	$u_{0}(t,x):$  in the interstitium.
	\item Permeabilities $([m/s])$: \\
	$P_{i}:$ between the lumen and the epithelium,  \\
	$P_{i,e}:$ between the epithelium and the interstitium.
\end{itemize}
In this work we will indicate as lumen the considered limb and as tubule the segment with its epithelial layer. In physiological common language, the term 'tubule' refers to the cavity of lumen together with its related epithelial layer (membrane) as part of it, \cite{laytonedwards}.\\
In the ascending tubule, the transport of solute both by  passive diffusion and  active re-absorption uses $Na^{+}/K^{+}$-ATPases pumps, which
exchange 3 $Na^{+}$ ions for 2 $K^{+}$ ions.
This active transport is modelled by a non-linear term given by the Michaelis-Menten kinetics~:
\begin{equation}\label{(G)}
G(q_2)=V_{m,2} \left (  \frac{q_2}{k_{M,2}+q_{2}} \right )^{3}.
\end{equation}
where  $k_{M,2}$ and $V_{m,2}$ are real positive constants. 
In each tubule, the fluid (mostly water) is assumed to flow at constant rate $\alpha$ and 
we only consider one generic uncharged solute in two tubules as depicted in Figure \ref{fig:1}.
\begin{figure}[ht!]	\centering
	\includegraphics[scale=0.5]{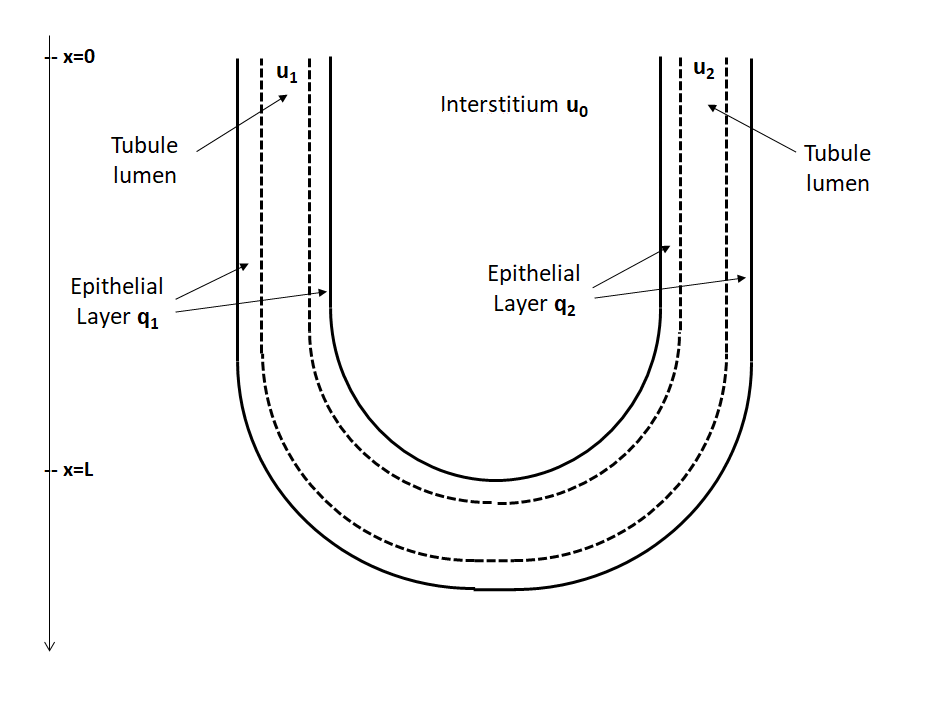}
	\caption{Simplified model of the loop of Henle. $q_1$, $q_2$,  $u_1$ and $u_2$ denote solute	concentration in the epithelial layer and lumen of the descending/ascending limb, respectively.}
	\label{fig:1}
\end{figure}

In a recent paper \cite{MAMV}, the authors have studied the role of the epithelial layer in the ionic transport. 
The aim of this work is to clarify the link between model \eqref{problem} taking into account the epithelial layer and models neglecting it. In particular, when the permeability between the epithelium and the lumen is large it is expected that these two regions merge, allowing to reduce system \eqref{problem} to  a model with no 
epithelial layer. 
More precisely,  as the permeabilities $P_1$ and $P_2$ grow large, 
we show rigorously that solutions of  \eqref{problem} with boundary conditions \eqref{problembord} 
relax to solutions of a reduced system with no epithelial layer.
From a mathematical point of view, system \eqref{problem} may be seen as a hyperbolic system with a stiff source term.
The source term is in some sense a relaxation of another hyperbolic system of smaller dimension. 
Such an approach has been widely studied in the literature, see e.g. \cite{JinXin,nataliniterracina,James,PST}.
Since the initial data of the starting system for fixed $\eps$ has no reason
to be compatible with the limit system, the mathematical analysis of this relaxation procedure should account for  initial layers.
In the setting of a generic relaxation problem concerning the Cauchy case for not 'well-prepared' data, or data out of equilibrium, the construction of initial layers and the corresponding error analysis can be found in \cite{GioYaYo}. 
The proof of our convergence result is obtained thanks to a $\BV$ compactness argument in space and in time.
Another difficulty is due to the presence of the boundary, which must be handled with care in the {\em a priori} estimates, in order to be uniform with respect to $\eps$, the relaxation parameter depending on the permeabilities.

\section{Main results}

Before presenting our main result, we list some assumptions which will be used throughout this paper.

\begin{assum} \label{ass.bvcond}
	We assume that the initial solute concentrations are non-negative and uniformly bounded in $L^\infty(0,L)$ and 
	with respect to the total variation~:
	\begin{equation}\label{bvcond}
	0\leq  u_{1}^{0}, u_{2}^{0}, q_{1}^{0}, q_{2}^{0}, u_{0}^{0} \in BV(0,L) \cap L^{\infty}(0,L).
	\end{equation}
\end{assum}
For  detailed definitions of the $\BV$ setting we refer to the standard text-books \cite{giusti,ziemer}.
A more recent overview gives a global picture in an extensive way \cite{HeiPaRe}, it unifies also 
the diversity of definitions found in the literature dealing with  the $\BV$ spaces in either
the probabilistic
or the deterministic context.

\begin{assum}\label{ass.boundcond}
	Boundary conditions are such that 
	\begin{equation}\label{boundcond}
	0 \leq u_b  \in BV(0,T) \cap L^{\infty}(0,T).
	\end{equation}
\end{assum}

\begin{assum}\label{ass.nonlin}
	Regularity and boundedness of $G$. \\
	We assume that the non-linear function modelling active transport in the ascending limb is an odd and  $W^{2,\infty}(\RR)$ function~:
	\begin{equation}\label{nonlin}
	\forall\, x\geq 0, \quad G(-x)=-G(x), \quad 0 \leq G(x) \leq \|G\|_\infty, \quad  0 \leq G'(x) \leq \|G'\|_\infty.
	\end{equation}
	We notice that the function $G$ defined on $\mathbb{R}^+$ by the expression in \eqref{(G)} may be straightforwardly extended by symmetry on $\mathbb{R}$ by a function satisfying \eqref{nonlin}.
\end{assum}

To simplify our notations in \eqref{problem}, we set $2\pi r_{i,e}P_{i,e}=K_i$, $i=1,2$ and $2\pi r_i P_i =k_i$ , $i=1,2$. The orders of magnitude of $k_1$ and $k_2$ are the same even if their values are not definitely equal, we may assume to further simplify the analysis that $k_1=k_2$.
We consider the case where permeability between the lumen and the epithelium is large and we set, $k_1=k_2=\frac{1}{\eps}$ for $\eps\ll 1$.
Then, we investigate the limit $\eps$ goes to zero of the solutions of  the following system~:
\begin{subequations}\label{five}
	\begin{align}
	\begin{split}\label{fivea}
	&\partial_{t}u_1^{\eps} + \alpha\partial_{x} u_1^{\eps}=\frac{1}{\eps}(q_1^{\eps}-u_1^{\eps}) 
	\end{split} \\
	\begin{split}\label{fiveb}
	&\partial_{t}u_2^{\eps} - \alpha\partial_{x} u_2^{\eps}=\frac{1}{\eps}(q_2^{\eps}-u_2^{\eps}) 
	\end{split} \\
	\begin{split}\label{fivec}
	&\partial_{t}q_1^{\eps}=\frac{1}{\eps}(u_{1}^{\eps}-q_1^{\eps}) + K_1(u_{0}^{\eps}-q_1^{\eps}) 
	\end{split} \\
	\begin{split}\label{fived}
	&\partial_{t}q_2^{\eps}=\frac{1}{\eps}(u_{2}^{\eps}-q_2^{\eps}) + K_2(u_{0}^{\eps}-q_2^{\eps})- G(q_2^{\eps})
	\end{split} \\
	\begin{split}\label{fivee}
	&\partial_{t}u_0^{\eps}=K_1(q_1^{\eps}-u_0^{\eps}) + K_{2}(q_2^{\eps}-u_0^{\eps}) + G(q_2^{\eps}) 
	\end{split}
	\end{align}
\end{subequations}

Formally, when $\eps \rightarrow 0 $, we expect the concentrations $u_1^{\eps}$ and $q_1^{\eps}$ to converge to the same function. The same happens for $u_2^{\eps}$ and $q_2^{\eps}$.
We denote $u_1$, respectively $u_2$, these limits.
Adding  \eqref{fivea} to \eqref{fivec}  and  \eqref{fiveb} to \eqref{fived}, 
we end up with the system
\begin{align*}
\p_{t} u_1^{\eps}+ \p_{t}q_1^{\eps}+ \alpha\p_{x} u_1^{\eps}= &\ K_1(u_{0}^{\eps}-q_1^{\eps}) \\
\p_{t} u_2^{\eps}+ \p_{t}q_2^{\eps}- \alpha\p_{x} u_2^{\eps}= &\ K_2(u_{0}^{\eps}-q_2^{\eps}) - G(q_2^{\eps}).
\end{align*}
Passing formally to the limit when  $\eps$ goes to 0, we arrive at
\begin{align}
2\p_{t} u_1+ \alpha\p_{x} u_1= &\ K_1(u_{0}-u_1)   \label{systlim1} \\
2\p_{t} u_2-\alpha\p_{x} u_2 = &\ K_2(u_{0}-u_2) - G(u_2), \label{systlim2}
\end{align}
coupled to the equation for the concentration in the interstitium obtained by passing into the limit in equation \eqref{fivee}
\begin{equation}\label{systlim3}
\p_{t} u_0=K_1(u_1-u_0) + K_{2}(u_2-u_0) + G(u_2).
\end{equation}
This system is complemented with the initial and boundary conditions
\begin{align}\label{initlim}
u_1(0,x) = u_1^0(x)+q_1^0(x), \quad u_2(0,x) = u_2^0(x)+q_2^0(x),
\quad u_0(0,x) = u_0^0(x), \\
u_1(t,0) = u_b(t), \quad u_2(t,L) = u_1(t,L). \label{bordlim}
\end{align}
Finally, we recover a simplified system for only three unknowns. From a physical point of view this means fusing the epithelial layer with the lumen. It turns out to merge the lumen and the epithelium into a single domain when we consider the limit of infinite permeability.
The aim of this paper is to make this formal computation rigorous.
For this sake, we define  weak solutions associated to the limit system \eqref{systlim1}-\eqref{systlim3} :
\begin{defi}\label{def.weak.solution.zero}
	Let $ u_{1}^{0}(x)$, $u_{2}^{0}(x)$, $u_{0}^{0}(x) \in L^1(0,L) \cap L^{\infty}(0,L)$ and $u_b(t) \in L^1(0,T)\cap L^{\infty}(0,T)$.
	We say that $U(t,x)=(u_1(t,x)$, $u_2(t,x)$, $u_0(t,x))^\top\in L^\infty((0,T);L^1(0;L)\cap L^\infty(0,L))^3$ is a weak solution of system \eqref{systlim1}-\eqref{systlim3}
	if for all $\bphi \in \cS_3$, with
	$$\mathcal{S}_3:=\left\{\bphi \in C^{1}([0,T]\times[0,L])^3, \quad \bphi(T,x)=0 ,\; \phi_1(t,L)=\phi_2(t,L) ,\text{ and }  \phi_2(t,0)=0\right\}, $$
	we have
	\begin{equation}\label{eq.weak.form.zero}
	\begin{aligned}
	&\int_{0}^{T} \int_{0}^{L} u_1(2 \p_t \phi_1+\alpha \p_x \phi_1) dxdt + \alpha \int_{0}^{T} u_b(t)\phi_1(t,0)\ dt
	+ \int_{0}^{L} u_1^0(x)\phi_1(0,x)\ dx \\
	&+  \int_{0}^{T} \int_{0}^{L} u_2( 2 \p_t \phi_2-\alpha \p_x \phi_2) 
	+ \int_{0}^{L} u_2^0(x)\phi_2(0,x)\ dx \\
	& + \int_{0}^{T} \int_{0}^{L} u_0 \p_t \phi_3  
	+K_1(u_1-u_0)  (\phi_3 -\phi_1 )
	+K_2 (u_2-u_0)  (\phi_3 -\phi_2 ) 
	+ G(u_2) (\phi_3- \phi_2)\ dxdt \\
	&  +\int_{0}^{L} u_0^0(x) \phi_3(0,x)\ dx=0. \\ 
	\end{aligned}
	\end{equation}
	
\end{defi}

More precisely, the main result reads
\begin{theo} \label{th:conv}
	Let $T>0$ and $L>0$. We assume that initial data and boundary conditions satisfy \eqref{bvcond}, \eqref{boundcond}, \eqref{nonlin}. Then, the weak solution $(u_1^{\eps},u_2^{\eps},q_1^{\eps},q_2^{\eps},u_0^{\eps})$ of system \eqref{five} with boundary and initial conditions \eqref{problembord} converges, as $\eps$ goes to zero, to the weak solution of reduced (or limit) problem \eqref{systlim1}--\eqref{systlim3} complemented with \eqref{initlim}--\eqref{bordlim}. More precisely,
	\begin{align*}
	& u_i^{\eps}  \xrightarrow[\eps \to 0]{} u_i \quad i=0,1,2, \quad \text{ strongly in } L^{1}([0,T]\times[0,L]), \\
	& q_j^{\eps}  \xrightarrow[\eps \to 0]{} u_j \quad j=1,2, \quad \text{ strongly in } L^{1}([0,T]\times[0,L]),
	\end{align*}
	where $(u_1,u_2,u_0)$ is the unique weak solution of the limit problem \eqref{systlim1}--\eqref{systlim3} 
	in the sense of Definition \ref{def.weak.solution.zero}.
\end{theo}
The system \eqref{five} can be seen as a particular case of the model without epithelial layer introduced and studied in \cite{tesp} and \cite{magali}. \\
A priori estimates uniform with respect to the parameter $\eps$ (accounting for permeability) are obtained in Section \ref{sec:estimates}.
We emphasize that estimates on time derivatives are more subtle due to specific boundary conditions of system and 
because one has to take care of singular initial layers.
Concerning existence and uniqueness of a solution, in previous works \cite{tesp} and \cite{magali}, 
authors proposed a semi-discrete scheme in space in order to show existence. 
In this work we propose  a fixed point theorem giving the same
result for any fixed $\eps > 0 $ in Section \ref{sec:existence theo}. The advantage of our approach is that we directly work with weak solutions
associated to \eqref{five}.
After recalling the definition of weak solution for problem \eqref{problem}, we report below the statement of Theorem \ref{th:exist}, 
and we refer to Section \ref{sec:existence theo} for the proof.

\begin{defi}\label{def.weak.solution}
	Let $ (u_{1}^{0}(x)$, $u_{2}^{0}(x)$, $q_{1}^{0}(x)$, $q_{2}^{0}(x)$, $u_{0}^{0}(x) )\in (L^1(0,L) \cap L^{\infty}(0,L))^5$ and $u_b(t) \in L^1(0,T)\cap L^{\infty}(0,T)$. Let $\eps>0$ be fixed.
	We say that $U^\e(t,x)=(u_1^\eps(t,x)$, $u_2^\eps(t,x)$, $q_1^\eps(t,x)$, $q_2^\eps(t,x)$, $u_0^\eps(t,x)) \in L^\infty((0,T);L^1(0,L)\cap L^\infty(0,L))^5$ is a weak solution of system \eqref{five} 
	if for all $\bphi=(\phi_1, \phi_2, \phi_3, \phi_4, \phi_5) \in \cS_5$,  with
	$$\cS_5:=\left\{\bphi \in C^{1}([0,T]\times[0,L])^5, \quad \bphi(T,x)=0  ,\; \phi_1(t,L)=\phi_2(t,L) ,\text{ and }  \phi_2(t,0)=0\right\}$$
	we have
	\begin{equation}\label{eq.weak.form.eps}
	\begin{aligned}
	&\int_{0}^{T} \int_{0}^{L} u_1^{\eps}(\p_t \phi_1+\alpha \p_x \phi_1) + \frac{1}{\eps}(q_1^{\eps}-u_1^{\eps}) \phi_{1}\ dxdt  + \alpha \int_{0}^{T} u_b^{\eps}(t)\phi_1(t,0)\ dt
	+ \int_{0}^{L} u_1^0(x)\phi_1(0,x)\ dx \\
	&+  \int_{0}^{T} \int_{0}^{L} u_2^{\eps}(\p_t \phi_2-\alpha \p_x \phi_2) +
	\frac{1}{\eps}(q_2^{\eps}-u_2^{\eps}) \phi_{2}\ dxdt 
	+ \int_{0}^{L} u_2^0(x)\phi_2(0,x)\ dx \\
	&+  \int_{0}^{T}  \int_{0}^{L} q_1^{\eps}(\p_t \phi_3) +K_1(u_0^{\eps}-q_1^{\eps}) \phi_3 - 
	\frac{1}{\eps}(q_1^{\eps}-u_1^{\eps})\phi_3\ dxdt 
	+\int_{0}^{L} q_1^0(x)\phi_3(0,x)\ dx  \\
	& + \int_{0}^{T} \int_{0}^{L} q_2^{\eps}(\p_t \phi_4) +K_2(u_0^{\eps}-q_2^{\eps}) \phi_4
	- \frac{1}{\eps}(q_2^{\eps}-u_2^{\eps}) \phi_4  - G(q_2^{\eps})\phi_4\ dxdt +
	\int_{0}^{L} q_2^0(x)\phi_4(0,x)\ dx \\ 
	& + \int_{0}^{T} \int_{0}^{L} u_0^{\eps}(\p_t \phi_5) +K_1(q_1^{\eps}-u_0^{\eps})  \phi_5 +K_2 (q_2^{\eps}-u_0^{\eps})  \phi_5 + G(q_2^{\eps}) \phi_5\ dxdt \\
	&  +\int_{0}^{L} u_0^0(x) \phi_5(0,x)\ dx=0. \\ 
	\end{aligned}
	\end{equation}
	
\end{defi}		

\begin{theo}[Existence]\label{th:exist}
	Under assumptions \eqref{bvcond}, \eqref{boundcond}, \eqref{nonlin} and for every fixed $\eps>0$, there exists a unique weak solution $U^\e$ of the problem \eqref{five}.
\end{theo}

\section{Proof of the existence result}\label{sec:existence theo}

		\newcommand{\bB}{{\boldsymbol{B} }}
	We define the Banach space $\bB := (L^1(0,L) \cap L^\infty(0,L))^5$.
	We prove existence using the Banach-Picard fixed point theorem (see e.g., \cite{perth} for various examples of its application). We consider a time $T>0$ (to be chosen later) and the map $\mT: X_T \rightarrow X_T$ with the Banach space 
	$X_T=L^{\infty}([0,T]; \bB)$, and we denote $\|\cdot\|_{X_T}=\sup_{t \in (0,T)} \|\cdot\|_{\bB}$.
	For a given function $\tU\in X_T$, with  $\tU= (\tu_1, \tu_2, \tq_1, \tq_2, \tu_0)$, we define $U:=\mT(\tU)$ solution to the  problem~:
	\begin{equation} \label{fiveexistence}
	\left\{ 
	\begin{aligned}
	&\partial_{t}u_1 + \alpha\partial_{x} u_1=\frac{(\tq_1-u_1)}{\eps},  \\
	&\partial_{t}u_2 - \alpha\partial_{x} u_2=\frac{(\tq_2-u_2)}{\eps}, \\
	&\partial_{t}q_1=\frac{(\tu_{1}-q_1)}{\eps} + K_1(\tu_{0}-q_1), \\
	&\partial_{t}q_2=\frac{(\tu_{2}-q_2)}{\eps} + K_2(\tu_{0}-q_2)- G(\tq_2), \\
	&\partial_{t}u_0=K_1(\tq_1-u_0) + K_{2}(\tq_2-u_0) + G(\tq_2),
	\end{aligned}
	\right.
	\end{equation}
	with initial data
	$u_{1}^{0}$, $u_{2}^{0}$, $q_{1}^{0}$, $q_{2}^{0}$, $u_{0}^{0}$ in $L^1(0,L) \cap L^{\infty}(0,L)$  and with boundary conditions
	$$
	u_{1}(t,0)=u_b(t)\geq 0\ , \quad u_{2}(t,L)=u_{1}(t,L), \quad \text{for } t > 0,
	$$
	where $u_b \in L^1(0,T) \cap L^{\infty}(0,T)$.
	
	First we define the solutions of \eqref{fiveexistence} using Duhamel's formula. 
	Under these hypothesis, we may compute $u_1$ and $u_2$ with the method of characteristics
	\begin{equation}\label{u1}
	u_1(t,x)=
	\begin{cases}
 	{u_1^0(x-\alpha t)e^{-\frac{t}{\eps}} +\frac{1}{\eps} \int_{0}^{t} e^{-\frac{t-s}{\eps}} \tq_1(x-\alpha(t-s),s)\ ds,} & \text{ if } x > \alpha t, \\
	{u_b\left(t-\frac{x}{\alpha} \right)e^{-\frac{x}{\eps \alpha}} +\frac{1}{\alpha \eps} \int_{0}^{x} e^{-\frac{1}{\alpha \eps}(x-y)} \tq_1\left(t-\frac{x-y}{\alpha},y\right)\ dy,} & \text{ if } x < \alpha t,
	\end{cases}
	\end{equation}
	with $u_1^0(x)$, and $u_b(t)$ initial and boundary condition, respectively.
	We have a similar expression for $u_2(t,x)$ with $u_2^{0}$ instead of $u_1^{0}$ and the boundary condition $u_2(t,L)=u_1(t,L)$ which  is well-defined thanks to \eqref{u1}. It reads :
	\begin{equation}\label{u2}
	u_2(t,x)=
	\begin{cases}
 	{u_2^0(x+\alpha t)e^{-\frac{t}{\eps}} +\frac{1}{\eps} \int_{0}^{t} e^{-\frac{t-s}{\eps}} \tq_2(s,x+\alpha(t-s))\ ds,} & \text{ if } x < L- \alpha t, \\
 u_1\left(t+\frac{x-L}{\alpha},L\right) e^{\frac{x-L}{\alpha \eps}} 	+ \frac{1}{\alpha \eps} \int_{x}^L e^{ \frac{x-y}{\e \alpha}} \tq_2(t+\frac{x-y}{\alpha},y) dy	& \text{ if } x>L-\alpha t.
	\end{cases}
	\end{equation}	
	Then,  for the other unknowns one simply solves a system of uncoupled ordinary differential equations leading to~:
	\begin{equation}\label{eq.ode}
	\left\{
	\begin{aligned}
	q_1(t,x)& = q_1^0(x)e^{-(\frac{1}{\eps}+K_1) t}+ \int_{0}^{t} e^{-(\frac{1}{\eps}+K_1)(t-s)} \Bigl(\frac{1}{\eps} \tu_1 +K_1 \tu_0 \Bigr)(s,x) \ ds, \\
	q_2(t,x)& = q_2^{0}(x)e^{-(\frac{1}{\eps}+K_2) t}+ \int_{0}^{t} e^{-(\frac{1}{\eps}+K_2)(t-s)} \Bigl(\frac{1}{\eps} \tu_2 +K_1 \tu_0 - G(\tq_2) \Bigr)(s,x) \ ds,\\
	u_0(t,x)& = u_0^{0}(x)e^{-(K_1+K_2) t}+ \int_{0}^{t} e^{-(K_1+K_2)(t-s)} \Bigl(K_1\tq_1 +K_2 \tq_2 +G(\tq_2) \Bigr)(s,x) \ ds.
	\end{aligned}
	\right.
	\end{equation}
	
	Using  regularity arguments as is Theorem 2.1 and Lemma 2.1 in \cite{MiOel.1}, one can show that thanks to the Lipschitz continuity of the solutions along the characteristics, 
	the previous unknowns solve the weak formulation reading~:
	\begin{equation}\label{eq.weak.duhamel}
	\left\{
	\begin{aligned}
	\int_0^T \int_0^L  - u_1 & \left( \dt + \alpha \dx  \right) \varphi_1(t,x) + \ue \left( u_1 - \tq_1 \right) \varphi_1(t,x) dx dt \\
	& + \left[ \int_0^L u_1(t,x) \varphi_1 (t,x)dt \right]_{t=0}^{t=T} + \alpha \left[ \int_0^T u_1(t,x ) \varphi_1(t,x)\right]_{x=0}^{x=L} = 0, \\
	\int_0^T \int_0^L  - u_2 & \left( \dt - \alpha \dx  \right) \varphi_2(t,x) + \ue \left( u_2 - \tq_2 \right) \varphi_2(t,x) dx dt \\
	& + \left[ \int_0^L u_2(t,x) \varphi_2 (t,x)dt \right]_{t=0}^{t=T} + \alpha \left[ \int_0^T u_2(t,x ) \varphi_2(t,x)\right]_{x=0}^{x=L} = 0, \\
	\end{aligned}
	\right.
	\end{equation}
	for any $(\varphi_1,\varphi_2) \in C^1([0,T]\times[0,L])^2$.
	Note that similar arguments as in Lemma 3.1 in \cite{MiOel.1} show that the same holds true for $|u_1|$ (resp. $|u_2|$ ) :
	\begin{equation}\label{eq.weak.abs}
	\begin{aligned}
	\int_0^T \int_0^L  - |u_1| & \left( \dt + \alpha \dx  \right) \varphi_1(t,x) + \ue \left( |u_1| - \sgn(u_1)\tq_1 \right) \varphi_1(t,x) dx dt \\
	& + \left[ \int_0^L |u_1|(t,x) \varphi_1 (t,x)dt \right]_{t=0}^{t=T} + \alpha \left[ \int_0^T |u_1|(t,x ) \varphi_1(t,x)\right]_{x=0}^{x=L} = 0.
	\end{aligned}
	\end{equation}
	The same holds also for the other unknowns $(q_i)_{i\in\{1,2\}}$ and $u_0$,
	since for the ODE part of the system \eqref{eq.ode} provides directly similar results.
	In the rest of the paper, each time that we mention 
	that we are multiplying formally by $\sgn $ each function of system \eqref{fiveexistence} in order to get :
	\begin{equation*}
	\left\{
	\begin{aligned}
	& \partial_{t}|u_1| + \alpha\p_{x} |u_1|= \frac{1}{\eps}( \sgn(u_1)\tq_1-|u_1|) \\
	&\partial_{t}|u_2| - \alpha\p_{x} |u_2|= \frac{1}{\eps}(\sgn(u_2)\tq_2-|u_2|) ,
	\end{aligned}
	\right.
	\end{equation*}
	we actually mean that these inequalities hold in the previous sense, {\em i.e.}  in the sense 
	of \eqref{eq.weak.abs}. 
	The reader should notice that the stronger regularity 
	of the integrated form \eqref{u1}, \eqref{u2} and \eqref{eq.ode}  allow
	to define the solutions on the boundaries of the domain $\partial ((0,T)\times(0,L))$.
	If these would only belong to $L^\infty ((0,T);L^1 (0,L)$ this would not make much sense.
	Now the meaning of the formal setting is well defined,  we then can proceed by writing that one has~:
	\begin{equation}\label{eq.ineq.formal}
	\left\{
	\begin{aligned}
	& \partial_{t}|u_1| + \alpha\p_{x} |u_1|\leq \frac{1}{\eps}(|\tq_1|-|u_1|) \\
	&\partial_{t}|u_2| - \alpha\p_{x} |u_2|\leq\frac{1}{\eps}(|\tq_2|-|u_2|) \\
	&\partial_{t}|q_1|\leq\frac{1}{\eps}(|\tu_{1}|-|q_1|) + K_{1}(|\tu_{0}|-|q_1|) \\
	&\partial_{t}|q_2|\leq\frac{1}{\eps}(|\tu_{2}|-|q_2|) + K_{2}(|\tu_{0}|-|q_2|)-|G(\tq_2)|\\
	&\partial_{t}|u_0|\leq K_{1}(|\tq_1|-|u_0|) + K_{2}(|\tq_2|-|u_0|) + |G(\tq_2)|. \\
	\end{aligned}
	\right.
	\end{equation}
	We have used the fact that $\sgn(G(\tq_2))=\sgn(\tq_2)$ from \eqref{nonlin}, which implies in particular $-G(\tq_2)\cdot$ $\sgn(\tq_2)= $ $-|G(\tq_2)|$ and $G(\tq_2)\cdot\sgn(u_0)\leq |G(\tq_2)|$.
	In order to obtain inequalities in the weak formulation associated to the latter system it is enough
	to choose non-negative test functions in $C^1([0,T]\times[0,L])$.
	
	Adding all equations and integrating on $[0,L]$, we obtain formally 
	\begin{align*}
	\frac{d}{dt}  \int_{0}^{L}& ( |u_1| + |u_2| + |u_0| + |q_1| + |q_2|) \leq \alpha |u_1(t,0)| \\
	& + \frac{1}{\eps}\int_{0}^{L} (|\tu_1| +|\tu_2|+|\tq_1|+|\tq_2|)\ dx + (K_1+K_2)\int_{0}^{L} |\tu_0|\ dx, 
	\end{align*}
	where we use the boundary condition $u_1(t,L)=u_2(t,L)$ and \eqref{nonlin}.
	Setting $\|U(t,\cdot)\|_{L^1(0,L)^5}:=\int_{0}^{L}(|u_1| + |u_2|  + |q_1| + |q_2|+|u_0|)(t,x)\ dx $ and integrating with respect to time, we obtain:
	\begin{equation}\label{estimU}
	\|U(t,x)\|_{L^1(0,L)^5}  \leq \|U(0,x)\|_{L^1(0,L)^5} + \alpha \int_{0}^{T}|u_b(s)|\ ds +\eta \int_{0}^{T} \|\tU(t,x)\|_{L^1(0,L)^5}\ dt,
	\end{equation}
	with $ \eta = K_1+K_2 +\frac{1}{\eps} >0 $.
	Here the formal computations are to be understood in the following manner : in the 
	weak formulation associated to \eqref{eq.ineq.formal} we choose the test function
	$\boldsymbol{\varphi} = (1,1,1,1,1)$, and the result \eqref{estimU} comes 
	in a straightforward way when neglecting the out-coming characteristic at $x=0$. 
	
	On the other hand, using \eqref{u1}, \eqref{u2} and \eqref{eq.ode}, one quickly checks that 
	\begin{equation}\label{eq.endo.linf}
	\nrm{U}{L^\infty((0,T)\times(0,L))^5} \leq 
	\max\left(\nrm{U^0}{L^\infty(0,L)^5},\nrm{u_b}{L^\infty(0,T)} \right) 
	+ \frac{C T}{\eps } \nrm{\tilde{U}}{L^\infty((0,T)\times(0,L))^5}
	\end{equation}
	where the generic constant $C$ depends only on $(K_i)_{i\in \{1,2\}}$ and $\nrm{G'}{L^\infty(\RR)}$
	but not on $\tilde{U}$ nor on the data $U^0$.
	At this step, $\mathcal{T}$ maps $X_T$ into itself.
	
	Let us now prove that $\mT$ is a contraction. 
	Let $(\tU,\tW) \in X_T^2$, we define $U:= \mT(\tU), \ W:=\mT(\widetilde{W})$.
	Then, by the same token as obtaining \eqref{estimU}, we have
	\begin{align*}
	\| \mathcal{T} (\tU) - \mathcal{T}(\tW) \|_{L^{1}(0,L)^5} =\| U-W \|_{L^1(0,L)^5}
	& \leq \eta \int_{0}^{T} \|\tU-\tW \|_{L^1(0,L)}  \\
	& \leq \eta T \| \tU-\tW \|_{X_T}.
	\end{align*}
	Again similar computations as in \eqref{eq.endo.linf}, show that
	$$
	\nrm{U-W}{L^\infty((0,T)\times(0,L))^5} \leq 
	\frac{C T}{\eps } \nrm{\tilde{U}-\tilde{W}}{L^\infty((0,T)\times(0,L))^5}.
	$$
	Therefore, as soon as $T < \min (1/\eta,\eps/C)$, $\mathcal{T}$ is a contraction in $X_T$.
	It allows to construct a solution on $[0,T]$ for $T$ small enough.
	The fixed point solves \eqref{u1} and \eqref{eq.ode} in an implicit way.
	Along characteristics  solutions have enough regularity to satisfy \eqref{five} 
	in a  weak sense \eqref{eq.weak.duhamel}. Choosing then the test functions 
	${\boldsymbol{\varphi}}:=(\varphi_i)_{i\in\{1,\dots,5\}}$ to belong to $\cS_5$ shows that
	the fixed point is a weak solution in the sense of Definition \ref{def.weak.solution}. 
	Since the solution $U(t,x)=(u_1(t,x),u_2(t,x),q_1(t,x),q_2(t,x),u_0(t,x))$ is well
	defined as on $\{T\}\times(0,L)$ thanks to regularity arguments stated above, 
	$U(T,x)$ becomes the initial condition of a new initial boundary problem.
	Thus, we may iterate this process on $[T, 2T]$, $[2T,3T]$, $\dots$, since the condition on $T$ does not depend on the iteration.
	
	As a result of above computations, we have also that if $U^1$ (resp. $U^2$) is a solution with initial data $U^{1,0}$ (resp. $U^{2,0}$) and boundary data $u_b^1$ (resp. $u_b^2$). Then we have the 
	comparison principle : 
	\begin{equation}\label{estim1}
	\| U^1 - U^2 \|_{L^1((0,T)\times(0,L))} \leq \| U^{0,1} - U^{0,2} \|_{L^1(0,L)} + \alpha \| u_b^1 - u_b^2 \|_{L^1(0,T)}.
	\end{equation}
	which shows and implies uniqueness as well.

\section{Uniform {\em a priori} estimates}\label{sec:estimates}
In order to prove our convergence result, we first establish some uniform {\em a priori} estimates. 
The strategy of the proof of Theorem \ref{th:conv} relies on a compactness argument. 
In this Section we will omit the index $\eps$ in order to simplify the notations.

\subsection{Non-negativity and $L^1\cap L^{\infty}$ estimates}

The following lemma establishes that all concentrations of system are non-negative and this is consistent with the biological framework.

\begin{lem}[Non-negativity]\label{positivity}
Let $U(t,x)$ be a weak solution of system \eqref{problem} such that the assumptions \eqref{bvcond}, \eqref{boundcond}, \eqref{nonlin} hold. Then for almost every  $(t,x) \in (0,T)\times(0,L),$ $ U(t,x)$ is non-negative, i.e.: 
$ u_1(t,x)$, $u_2(t,x)$, $q_1(t,x)$, $q_2(t,x)$, $u_0(t,x)$ $\geq 0. $ 
\end{lem}


\begin{proof} We prove that the negative part of our functions vanishes. Using Stampacchia's method, we formally multiply each equation of system \eqref{five} by corresponding indicator function as follows:
\begin{equation}
\begin{cases}
(\p_{t}u_1 + \alpha\p_{x} u_1) \induu =\frac{1}{\eps}(q_1-u_1) \induu\\
(\p_{t}u_2 - \alpha\p_{x} u_2) \indud=\frac{1}{\eps}(q_{2}-u_2) \indud \\
(\p_{t}q_{1})\indvu=\frac{1}{\eps}(u_{1}-q_1)\indvu +K_{1}(u_{0}-q_1) \indvu \\
(\p_{t}q_{2})\indvd=\frac{1}{\eps}(u_{2}-q_{2})\indvd + K_{2}(u_{0}-q_{2})\indvd- G(q_{2})\indvd\\
(\p_{t}u_0)\induz=K_{1}(q_1-u_0)\induz + K_{2}(q_{2}-u_0)\induz+G(q_{2})\induz.
\nonumber
\end{cases}
\end{equation}
again as in the proof of existence in Section \ref{sec:existence theo}, these computations can be made rigorously 
using the extra regularity provided along characteristics in the spirit of Lemma 3.1, \cite{MiOel.1}.

We remember that for each function $u$ we can define positive and negative parts as  $u^{+}=\max(u,0)$, $u^{-}=\max(-u,0)$.
One has obviously that $u_i^{-}= -u_i \mathbf{1}_{\{u_{i}<0\}}$ for any $u \in L^1_{\rm loc}((0,T)\times(0,L))$,
whereas along characteristics curves one has du to Lipschitz continuity of the solutions that
It is possible also to write in the distributional sense:
$$
\p_t u_i^{-}=-\p_{t} u_{i} \mathbf{1}_{\{u_{i}<0\}} \quad i=0,1,2. 
$$
We refer again to Lemma 3.1 \cite{MiOel.1} for more detailed  explanations.
The same is true for other functions $q_j $ with $j=1,2$. 

Taking into account the fact that:
$$ q_i \mathbf{1}_{\{u_{i}<0\}}=  (q_{i}^{+}-q_{i}^{-}) \mathbf{1}_{\{u_{i}<0\}} \geq -q_{i}^{-}, \quad i=1,2, $$
since $q_{i}^{-} \mathbf{1}_{\{u_{i}<0\}}$ is zero or positive by definition of negative part, we obtain :  
$$
\begin{cases}
\p_{t}u_1^{-} +\alpha\p_{x} u_1^{-} \leq \frac{1}{\eps}(q_1^{-}-u_1^{-}) \\
\p_{t}u_2^{-} - \alpha\p_{x} u_2^{-}\leq\frac{1}{\eps}(q_{2}^{-}-u_2^{-})  \\
\p_{t}q_{1}^{-}\leq\frac{1}{\eps}(u_{1}^{-}-q_1^{-})+K_{1}(u_{0}^{-}-q_1^{-}) \\
\p_{t}q_{2}^{-}\leq\frac{1}{\eps}(u_{2}^{-}-q_{2}^{-})+ K_{2}(u_{0}^{-}-q_{2}^{-})+ G(q_{2})\indvd \\
\p_{t}u_0^{-}\leq K_{1}(q_1^{-}-u_0^{-}) + K_{2}(q_{2}^{-}-u_0^{-})-G(q_{2})\induz.
\end{cases}
$$
Adding the previous expressions, one recovers a single inequality reading
$$
\p_t (u_1^{-} +  q_{1}^{-}+  q_{2}^{-}	+ u_2^{-} + u_0^{-}) +\alpha \p_x (u_1^{-}- u_2^{-})  \leq G(q_2)(\indvd - \induz).
$$

By Assumption \ref{ass.nonlin}, we have that $\sgn(G(q_2))=\sgn(q_2)$. Thus $G(q_2)(\indvd - \induz) =$ 
$ G(q_2)$ $(\mathbf{1}_{\{G(q_{2})<0\}} - \induz) \leq 0$.
Then integrating on the interval $[0,L]$, we get :
$$
\frac{d}{dt}\int_{0}^{L} (u_1^{-} +q_{1}^{-}+ q_{2}^{-}+u_2^{-}+u_0^{-})(t,x)\ dx \leq \alpha ( u_2^{-}(t,L)- u_2^{-}(t,0) - u_1^{-}(t,L) + u_1^{-}(t,0)).
$$
Since $u_{1}^{-}(t,L)=u_{2}^{-}(t,L) $ thanks to condition \eqref{boundcond}, it follows:			
$$
\frac{d}{dt}\int_{0}^{L}( u_1^{-} +q_{1}^{-}+ q_{2}^{-}+u_2^{-}+u_0^{-})(t,x)\ dx \leq \alpha u_1^{-}(t,0) = \alpha u_b^{-}(t).
$$
From Assumptions \ref{ass.boundcond} and \ref{ass.bvcond}, the initial and boundary data are 
all non-negative.
Thus $ u_1^{-}(0,x)$,$q_{1}^{-}(0,x)$, $q_{2}^{-}(0,x)$,	$u_2^{-}(0,x)$,$u_0^{-}(0,x)$ are necessarily  zero.
This proves solutions' non-negativity and  concludes the proof.
\end{proof}


\begin{lem}[$L^\infty$ bound]\label{lem:infty}
Let $(u_1, u_2, q_1, q_2, u_0)$ be the unique weak solution of problem \eqref{five}. 
Assume that \eqref{bvcond}, \eqref{boundcond}, \eqref{nonlin} hold, then it is bounded {\em i.e.} 
for a.e. $(t,x)\in (0,T)\times (0,L)$,
$$
0 \leq u_{0}(t,x)\leq \ka(1+t),  \quad 0 \leq u_{i} (t,x)\leq \ka(1+t), \quad0 \leq q_{i} (t,x)\leq \ka(1+t), \quad i=1,2,
$$
$$
0 \leq u_2(t,0) \leq \ka(1+t), \quad 0\leq u_1(t,L) \leq \ka(1+t),
$$
where the constant $\ka \geq \max\left\{ \|G\|_{\infty}, \|u_b\|_{\infty}, \|u_0^0\|_{\infty}, \|u_i^0\|_{\infty}, \|q_i^0\|_{\infty}, i\in\{1,2\}   \right\}$.
\end{lem}

\begin{proof}
We use the same method as in the previous lemma for the functions
$$
w_i=(u_i-\ka(1+t)), \quad i=0,1,2, \quad z_j=(q_j-\ka(1+t)), \quad j=1,2.
$$
From system \eqref{five} and using the fact that 
$$ z_j{\mathbf{1}_{\{w_{i}\geq 0\}}}=z_j^{+}{\mathbf{1}_{\{w_{i}\geq 0\}}}-z_j^{-}{\mathbf{1}_{\{w_{i}\geq 0\}}}
\leq z_j^{+}, \quad\quad w_i{\mathbf{1}_{\{z_{j}\geq 0\}}}\leq w_i^{+},$$ we get
\begin{equation}\label{systLinf}
\left\{
\begin{aligned}
&\p_{t}w_1^{+} +\ka\indw +\alpha\p_{x} w_1^{+}\leq \frac{1}{\eps}(z_1^{+}-w_1^{+}) \\
&\p_{t}w_2^{+} +\ka\indww - \alpha\p_{x} w_2^{+}\leq \frac{1}{\eps}(z_2^{+}-w_2^{+}) \\
&\p_{t}z_1^{+}+\ka\indz \leq \frac{1}{\eps}(w_{1}^{+}-z_1^{+}) + K_1(w_{0}^{+}-z_1^{+}) \\
&\p_{t}z_2^{+}+\ka\indzz \leq \frac{1}{\eps}(w_{2}^{+}-z_2^{+}) + K_2(w_{0}^{+}-z_2^{+})- G(q_2)\indzz \\
&\p_{t}w_0^{+}+\ka\indwz \leq K_1(z_1^{+} -w_0^{+}) + K_{2}(z_2^{+}-w_0^{+}) + G(q_2)\indwz.
\end{aligned}	
\right.
\end{equation}
Adding expressions above gives 
\begin{align*}
\p_{t}(w_1^{+} + w_2^{+} + z_1^{+} + z_2^{+}) +\alpha \p_{x} (w_1^{+}- w_2^{+}) \leq
\  -\ka \indwz 
\ + G(q_2)(\indwz-\indzz).
\end{align*}
Integrating with respect to $x$ yields
\begin{align*}
&\frac{d}{dt}\int_{0}^{L} (w_1^{+} +w_2^{+} +z_1^{+} +z_2^{+}+w_{0}^{+})(t,x)\,dx \\
& \leq  \alpha( w_2^{+}(t,L)-w_{2}^{+}(t,0)- w_1^{+}(t,L)+w_1^{+}(t,0))  
+ \int_{0}^{L} (G(q_2) - \ka) \indwz\,dx,
\end{align*}
where we use the fact that $G(q_2)\geq 0$ from assumption \eqref{nonlin} since $q_2\geq 0$ thanks to the previous lemma.
From the boundary conditions in \eqref{problem}, we have for all $t\geq 0$, 
$w_2^{+}(t,L)=[u_2(t,L)- \ka (1+t)]^{+}=[u_1(t,L)- \ka (1+t)]^{+}=w_1^{+}(t,L)$.
Then,
\begin{align*}
\frac{d}{dt}\int_{0}^{L} (w_1^{+} +w_2^{+} +z_1^{+} +z_2^{+}+w_{0}^{+})(t,x)\,dx
+ \alpha w_2^+(t,0) \\
\leq 
\alpha (u_b(t)-\kappa(1+t))^{+} 
+ (\|G\|_{\infty}-\ka)\int_{0}^{L} \indwz \,dx.
\end{align*}
If we adjust the constant $\ka$ such that $\ka \geq \max\left\{ \|G\|_{\infty}, \|u_b\|_{\infty} \right\}$, it implies that~:

$$ \frac{d}{dt}\int_{0}^{L} (w_1^{+} +w_2^{+} +z_1^{+} +z_2^{+}+w_{0}^{+})(t,x)\,dx
+ \alpha w_2^+(t,0) \leq 0,
$$
which shows the claim.

For the last estimate on $u_1(t,L)$, we sum the first and the third inequalities of the system \eqref{systLinf} and integrate on $(0,L)$,
\begin{align*}
\frac{d}{dt} \int_0^L (w_1^+ + z_1^+)\,dx + \alpha w_1^+(t,L) \leq\ 
& \alpha w_1^+(t,0) + K_1 \int_0^L (w_0^+ - z_1^+)\,dx \\
& - \ka \int_0^L(\indw+\indz)\,dx.
\end{align*}
Integrating on $(0,T)$ and since we have proved above that $w_0^+=0$ and $z_1^+=0$, we arrive at
$$
\alpha \int_0^T w_1^+(t,L)\,dt \leq \alpha \int_0^T w_1^+(t,0)\,dt = 0,
$$
for $\ka\geq \|u_b\|_\infty$.
\end{proof}

\begin{lem}[$L^\infty_t L^1_x$ estimates]\label{prop}
Let $T>0$ and let $(u_1,u_2,q_1,q_2,u_0)$ be a weak solution of system \eqref{problem} in
$\big(L^{\infty}([0,T]; (L^1 \cap L^{\infty})(0,L))\big)^5.$
We define:
$$\mathcal{H}(t)=\int_{0}^{L} (|u_1|+|u_2|+|u_0|+|q_1|+|q_2|)(t,x)\ dx.$$
Then, under hypothesis \eqref{bvcond}, \eqref{boundcond}, \eqref{nonlin} the following a priori estimate, uniform in $\eps>0$, holds:
$$
\mathcal{H}(t) 
\leq \alpha \|u_b\|_{L^{1}(0,T)} +\mathcal{H}(0), \quad \forall t > 0.
$$
Moreover the following inequalities hold:
\begin{align*}
\int_{0}^{T} |u_2(t,0)|\,dt \leq \|u_b\|_{L^{1}(0,T)} + \frac{1}{\alpha}\mathcal{H}(0),
\end{align*}
and
\begin{align*}
\int_{0}^{T} |u_1(t,L)|\,dt \leq \int_{0}^{L} (|u_1^{0}(x)|+|q_1^{0}(x)|)\,dx + CT 
\end{align*}
with $C>0$ constant.
\end{lem}
\begin{proof}
Since from Lemma \ref{positivity} all concentrations are non-negative, we may write from system \eqref{five}
\begin{equation}\label{mod}
\begin{cases}
\partial_{t}|u_1| + \alpha\p_{x} |u_1| = \frac{1}{\eps}(|q_1|-|u_1|) \\
\partial_{t}|u_2| - \alpha\p_{x} |u_2| = \frac{1}{\eps}(|q_2|-|u_2|) \\
\partial_{t}|q_1| = \frac{1}{\eps}(|u_{1}|-|q_1|) + K_{1}(|u_{0}|-|q_1|) \\
\partial_{t}|q_2| = \frac{1}{\eps}(|u_{2}|-|q_2|) + K_{2}(|u_{0}|-|q_2|)- |G(q_2)|\\
\partial_{t}|u_0| = K_{1}(|q_1|-|u_0|) + K_{2}(|q_2|-|u_0|) + |G(q_2)|.
\end{cases}
\end{equation}
Adding all equations and integrating on $(0,L)$, we get, recalling the boundary condition $u_1(t,L)=u_2(t,L)$,
\begin{equation}\label{est}
\frac{d}{dt}\mathcal{H}(t) + \alpha|u_{2}(t,0)| = \alpha |u_1(t,0)|
= \alpha |u_b(t)|.
\end{equation}
Integrating now with respect to time, we obtain:
\begin{equation}\label{estimate}
\mathcal{H}(t) + \alpha \int_0^t |u_{2}(s,0)|\,ds \leq \alpha \int_{0}^{t}|u_b(s)|\ ds +\mathcal{H}(0). 
\end{equation}
with $\mathcal{H}(t)$ previously defined. It gives the first two estimates of the Lemma.
Finally, to obtain the last inequality, we add equations \eqref{fivea} and \eqref{fivec} and integrate on $(0,L)$ to get
$$
\frac{d}{dt}\int_{0}^{L}(|u_1|+ |q_{1}|)\ dx +\alpha|u_1(t,L)| \leq \alpha|u_{b}(t)| + K_{1}\int_{0}^{L}|u_{0}|\ dx.
$$
Since we have shown that $\int_{0}^{L}|u_{0}|\ dx \leq \mathcal{H}(t) <\infty$,
we can conclude after integrating with respect to time.

\end{proof}

\subsection{Estimates on the derivatives}

\subsubsection{Data regularization}

Here we detail the notion of regularization for $\BV$ functions. 
The regularization denoted $f_\delta$ for a generic
$BV(0,L)$ function $f$ is described in the proof of Theorem 5.3.3 \cite{ziemer}. 
It provides the estimates from above :
$$
\nrm{\dx f_\delta}{L^1(0,L)} \leq \nrm{f}{BV(0,L)}.
$$
Using the standard mollifier this result is not true as stated p. 225 \cite{ziemer}, 
since $BV$ space is not separable.

\begin{defi}\label{def.reg.data}
If $(u_1^0,u_2^0,q_1^0,q_2^0,u_0^0)$ and $u_b$ are respectively the initial and
boundary data associated to the problem \eqref{five},
under hypotheses \ref{bvcond} and \ref{boundcond}, we define as {\bf regular data}
their regularization in the following manner~: we set
$$
\left\{
\begin{aligned}
\uzdu(x) & := ((1-\chiu{\delta}(x)-\chi_\delta(L-x))u_1^0+ c_1 \chiu{\delta}(x)  + c_2 \chi_{\delta}(L-x))_\delta  ,& \forall x\in[0,L] \\
\ubd (t) & := ((1-\chiu{\delta}(t))u_b+ c_1 \chiu{\delta}(t) )_\delta ,&  \forall t\in[0,T] \\
\uzdd(x) & := ((1-\chiu{\delta}(L-x))u_2^0 +c_2 \chiu{\delta}(L-x))_\delta , &  \forall x\in[0,L]\\
\qzdu(x) & := (q_1^0)_\delta , &  \forall x\in[0,L]\\
\qzdd(x) & := (q_2^0)_\delta , &  \forall x\in[0,L]\\
\uzdz(x) & := (u_0^0)_\delta , &  \forall x\in[0,L]\\
\end{aligned}
\right.
$$
where the regularization procedure is extracted from the proof of Theorem 5.3.3 \cite{ziemer}
and we define
$$
\chiu{\delta}(t) := 
\chi\left( \frac{t}{\delta} \right) ,\quad \chi(t) :=
\begin{cases}
1 & \text{if} \; |t|<1 \\
0 & \text{if} \; |t|>2
\end{cases}
$$
where $\chi\in C^\infty(\RR)$ is a positive monotone function.   
\end{defi}
\noindent On the other hand, we introduced  arbitrary constants $(c_i)_{i\in\{1,2\}}$ such that
\begin{itemize}
\item that match the initial and boundary condition on $x=0$ 
for the incoming characteristic.
\item that prevents mismatches between $u_1^0$, $u_2^0$ and the boundary
condition  $u_2(t,L)=u_1(t,L)$ in the neighbourhood of $x=L$.
\end{itemize}
The matching is $C^\infty$ in the neighbourhood of $(0,0)$ in  $[0,L]\times[0,T]$. Indeed, $\uzdu(x)= c_1$ when $x$ is close enough to $0$ and
in the same way $\ubd(t)=c_1$ when $t$ is near $0$, whereas for the derivatives
$(\uzdu)^{(k)} (x) = 0$ when $x$ is close to zero for any derivative of order $k$, and the same holds for $(\ubd)^{(k)}(t)$
when $t$ is sufficiently small. The same holds true in the neighbourhood of the point  $(t,x)=(L,0)$.

This regularization procedure allows then to obtain
\begin{lem}\label{lem.reg.reg}
Assume hypotheses \ref{ass.nonlin}, and let $U^\delta$ be the solution  associated to problem \eqref{five} with initial  data $U^{0,\delta}=(u_1^{0,\delta},u_2^{0,\delta},q_1^{0,\delta},q_2^{0,\delta},u_0^{0,\delta})$ and the boundary condition $\ubd$. Then $\dt U^\delta$ belongs to $X_T:=L^\infty((0,T);(L^1(0,L)\cap L^\infty(0,L)))^5$.
and solves the problem 
\begin{equation}\label{eq.dt.U}
\left\{
\begin{aligned}
& (\dt + \alpha \dx) u_{1,t}^\delta = \ue \left( q_{1,t}^\delta - u_{1,t}^\delta \right) \\
& (\dt - \alpha \dx) u_{2,t}^\delta = \ue \left( q_{2,t}^\delta - u_{2,t}^\delta \right) \\
& \dt q_{1,t}^\delta = - \ue \left( q_{1,t}^\delta - u_{1,t}^\delta \right)  + K_1(u_{0,t}^\delta-q_{1,t}^\delta)\\
& \dt q_{2,t}^\delta = - \ue \left( q_{2,t}^\delta - u_{2,t}^\delta \right) + K_2(u_{0,t}^\delta-q_{2,t}^\delta) - G'(q_2^\delta) q_{2,t}^\delta\\ 
& \dt u_{0,1}^\delta =  K_1(q_{1,t}^\delta - u_{0,t}^\delta) + K_2(q_{2,t}^\delta - u_{0,t}^\delta- G'(q_2^\delta) q_{2,t}^\delta
\end{aligned}
\right.
\end{equation}
where $u_{i,t}^\delta = \dt u_i^\delta$ and so on.
\begin{equation}\label{eq.bdr.cond.dt.U}
\left\{
\begin{aligned}
u_{1,t}^\delta(t,0) & = \dt u^\delta_b(t) , \\ 
u_{1,t}^\delta(0,x) & = - \alpha \dx u_1^{0,\delta} + \ue \left(q_{1}^{0,\delta}-u_1^{0,\delta}\right) \\
u_{2,t}^\delta(0,x) & =  \alpha \dx u_2^{0,\delta} + \ue \left(q_{2}^{0,\delta}-u_2^{0,\delta}\right) \\
q_{1,t}^\delta(0,x) & =   - \ue \left(q_{1}^{0,\delta}-u_1^{0,\delta}\right)+ K_1(u_{0}^{0,\delta}-q_{1}^{0,\delta})  \\
q_{2,t}^\delta(0,x) & =   - \ue \left(q_{2}^{0,\delta}-u_2^{0,\delta}\right) + K_2(u_{0}^{0,\delta}-q_{2}^{0,\delta}) - G(q_2^{0,\delta})  \\
u_{0,t}^\delta(0,x)
& = K_1(q_{1}^{0,\delta}- u_0^{0,\delta }) + K_2 (q_2^{0,\delta }-u_0^{0,\delta})  + G (q_2^{0,\delta})
\end{aligned}
\right.
\end{equation}

\end{lem}

\begin{proof}
The Duhamel's formula obtained by the fixed point method in the proof of Theorem \ref{th:exist} provides
a solution $U^\delta \in X_T$. Deriving $U^\delta$ with respect to $t$, one can show that $\dt U^\delta$
solves \eqref{eq.dt.U} with  initial and boundary conditions \eqref{eq.bdr.cond.dt.U}.
Applying then the existence result again proves that actually $\dt U^\delta$ belongs to $X_T$.
\end{proof}

\begin{remark}
 {\em A priori } estimates from  previous sections, when applied to the problem \eqref{eq.dt.U} complemented with initial-boundary data \eqref{eq.bdr.cond.dt.U}, do not provide  a control of   $\dt U^\delta$ which is 
uniform with respect to $\e$.
\end{remark}

This remark motivates next paragraphs.

\newcommand{\bW}{{\mathbf W}}

\subsubsection{The initial layer}\label{sec.init.layer}

When $\eps$ goes to zero, the concentrations $u_1, q_1$ and $u_2, q_2$ approach very quickly each other becoming roughly speaking the same. They relax turning out to be equal exponentially fast in time. 
When considering the time derivative of our unknowns
this fast convergence  provides a singular contribution
to the estimates. In order to account for this phenomenon, we introduce initial layer correctors. \\
On the microscopic scale we define for $t \in \RR_+$, the initial layer correctors 
$(\ti{u}_1,\ti{u}_2,\ti{q}_1,\ti{q}_2)$ solving 
\begin{equation}
\begin{cases}
\p_t \tilde{u}_1= \tilde{q}_{1}-\tilde{u}_{1} \\
\p_t \tilde{u}_2= \tilde{q}_{2}-\tilde{u}_{2} \\
\p_t \tilde{q}_1= \tilde{u}_{1}-\tilde{q}_{1} \\
\p_t \tilde{q}_2= \tilde{u}_{2}-\tilde{q}_{2},
\end{cases}
\end{equation}		 				
with initial conditions 
$$
\tilde{u}_{1}(0,x)=q_{1}^{0,\delta}-u_{1}^{0,\delta}, \quad \tilde{u}_{2}(0,x)=q_{2}^{0,\delta}-u_{2}^{0,\delta},\quad \tilde{q}_1(0,x)=0,\quad  \tilde{q}_2(0,x)=0.
$$
Actually, this system may be solved explicitly and we obtain
\begin{equation}\label{eq.tilde}
\tu_{i}(t,x)=\frac{1}{2}(q_{i}^{0,\delta}(x)-u_{i}^{0,\delta}(x))(1+e^{-2t}); 
\quad   \tq_{i}(t,x)=
\frac{1}{2}(q_{i}^{0,\delta}(x)-u_{i}^{0}(x))(1-e^{-2t}), \quad  i=1,2.
\end{equation}
\subsubsection{Uniform $L^1$ bounds of  the time derivatives}\label{sec.bound.time}
We introduce the following quantities on the macroscopic time scale $t\in[0,T]$ :
\begin{equation}\label{fundef}
\begin{cases}
v_1^\delta(t,x)= u_1^\delta(t,x)+\tilde{u}_1(\frac{t}{\eps},x) \\
v_2^\delta(t,x)= u_2^\delta(t,x)+\tilde{u}_2(\frac{t}{\eps},x)  \\
r_1^\delta(t,x)= q_1^\delta(t,x)+\tilde{q}_1(\frac{t}{\eps},x)  \\
r_2^\delta(t,x)= q_2^\delta(t,x)+\tilde{q}_2(\frac{t}{\eps},x) 
\end{cases}
\end{equation}	
Next, we prove  uniform bounds on the time derivatives :
\begin{prop}\label{propdt}
Let $T>0$. If the data is regular in the sense of Definition \ref{def.reg.data}, setting : 
$$
\tilde{\mathcal{H}}_{t}(t)=
\int_{0}^{L}(|\p_t v^\delta_{1}|+|\p_t v^\delta_{2}|+ |\p_t r^\delta_{1}|+|\p_t r^\delta_{2}|+|\p_t u^\delta_{0}|)(t,x)\ dx,
$$
with functions $v_1,v_2,v_0,r_1,r_2$  defined in \eqref{fundef}, one has :
\begin{equation}\label{estimU1t}
\begin{aligned}
\tilde{\mathcal{H}}_{t}(t) 
+  \int_0^t & 
|\p_t v^\delta_2(\tau,0)|\,d\tau
+ \int_0^t |\p_t v^\delta_1(\tau,L)|\,d\tau \\
& \leq C \left(
\nrm{U^{0,\delta}}{\bW^{1,1}(0,L)} + \nrm{u_b^\delta}{W^{1,1}(0,T)}\right),\quad \text{ for a.e. }t\in (0,T),
\end{aligned} 
\end{equation}
where $\bW^{1,1} (0,L)$ denotes the vector-space $W^{1,1}(0,L)^5$. 
\end{prop}						
\begin{proof}
	From system \eqref{five} we deduce
	\begin{equation}\label{funsist}
	\begin{cases}	
	\p_t v_1^\delta + \p_x v_1^\delta=\frac{1}{\eps} (r_1^\delta-v_1^\delta) + \p_x \tilde{u}_1(\frac{t}{\eps},x) \\
	\p_t v_2^\delta - \p_x v_2^\delta=\frac{1}{\eps} (r_2-v_2^\delta) - \p_x \tilde{u}_2(\frac{t}{\eps},x) \\
	\p_t r_1^\delta =\frac{1}{\eps} (v_1^\delta-r_1^\delta)+K_1(u_0^\delta-r_1^\delta) +K_1 \tilde{q}_1(\frac{t}{\eps},x) \\
	\p_t r_2 =\frac{1}{\eps} (v_2^\delta-r_2)+K_2(u_0^\delta-r_2) +K_2 \tilde{q}_2(\frac{t}{\eps},x)- G(q_2^\delta) \\
	\p_t u_0^\delta =K_1(r_1^\delta-u_0^\delta) +K_2 (r_2-u_0^\delta)- K_1 \tilde{q}_1(\frac{t}{\eps},x)-K_2\tilde{q}_2(\frac{t}{\eps},x) + G(q_2^\delta)
	\end{cases}				
	\end{equation}	
	with following initial and boundary conditions:					
	\begin{equation}\label{incond2}
	\begin{array}{ll}
	v_1^\delta(t,0)=u_1(t,0)+\tilde{u}_1(t,0)=u_b(t)+\tilde{u}_1(\frac{t}{\eps},0), & t \in (0,T), \\
	v_2^\delta(t,L) = v_2^\delta(t,L) + \tilde{u}_2(\frac{t}{\eps},L), & t \in (0,T),  \\
	v_1^\delta(0,x)=u_1(0,x)+\tilde{u}_1(0,x)=q_{1}^{0}(x), & x \in (0,L), \\
	v_2^\delta(0,x)=u_2(0,x)+\tilde{u}_2(0,x)=q_{2}^{0}(x), & \\
	r_1^\delta(0,x)=q_1(0,x)+\tilde{q}_1(0,x)=q_{1}^{0}(x), & \\
	r_2(0,x)=q_2^\delta(0,x)+\tilde{q}_2(0,x)=q_{2}^{0}(x). & \\
	\end{array}
	\end{equation}
	As $G \in C^2(\RR)$, thanks to Lemma \ref{lem.reg.reg}, $\dt U^\delta \in L^\infty((0,T);(L^1(0,T)\cap L^\infty(0,L)))^5$, 
	taking the derivative with respect to $t$ in system \eqref{funsist}, $\dt V^\delta =\dt ( v_1^\delta, v_2^\delta,r_1^\delta,r_2^\delta,u_0^\delta)$ solves
	\begin{equation*}\label{funsistdt}
	\begin{cases}
	\p_t v_{1,t}^\delta + \p_x v_{1,t}^\delta=\frac{1}{\eps} (r_{1,t}^\delta-v_{1,t}^\delta)+\frac{1}{\eps}\p_x \tilde{u}_{1,t} \\
	\p_t v_{2,t}^\delta - \p_x v_{2,t}^\delta=\frac{1}{\eps} (r_{2,t}^\delta-v_{2,t}^\delta)-\frac{1}{\eps}\p_x \tilde{u}_{2,t} \\
	\p_t r_{1,t}^\delta =\frac{1}{\eps} (v_{1,t}^\delta-r_{1,t}^\delta)+K_1(u_{0,t}^\delta-r_{1,t}^\delta) +\frac{1}{\eps}K_1 \tilde{q}_{1,t} \\
	\p_t r_{2,t}^\delta =\frac{1}{\eps} (v_{2,t}^\delta-r_{2,t}^\delta)+K_2(u_{0,t}^\delta-r_{2,t}^\delta) + \frac{1}{\eps} K_2 \tilde{q}_{2,t}- G'(q_2^\delta)q_{2,t} \\
	\p_t u_{0,t}^\delta =K_1(r_{1,t}^\delta-u_{0,t}^\delta) +K_2 (r_{2,t}^\delta-u_{0,t}^\delta)- 
	\frac{1}{\eps}K_1 \tilde{q}_{1,t}-\frac{1}{\eps}K_2\tilde{q}_{2,t}+ G'(q_2^\delta)q_{2,t}
	\end{cases}				
	\end{equation*}	
	in the sense of Definition \eqref{def.weak.solution}.			
	Again formally, we   multiply each equation respectively by $\sgn(v^\delta_{i,t})$ with $i=1,2$, and  $\sgn(r^\delta_{j,t})$, for $j=1,2$, and by $\sgn(u_{0,t}^\delta)$ in the sense explained in the proof of Theorem \ref{th:exist}. This gives 
	\begin{equation}\label{systemUQdt}
	\begin{cases}	
	\p_t |v_{1,t}^\delta| + \p_x |v_{1,t}^\delta| \leq \frac{1}{\eps} (|r_{1,t}^\delta|-|v_{1,t}^\delta|)+|\frac{1}{\eps}\p_x \tilde{u}_{1,t}| \\
	\p_t |v_{2,t}^\delta| - |\p_x v_{2,t}^\delta| \leq \frac{1}{\eps} (|r_{2,t}^\delta|-|v_{2,t}^\delta|)+|\frac{1}{\eps}\p_x \tilde{u}_{2,t}| \\
	\p_t |r_{1,t}^\delta|  \leq \frac{1}{\eps} (|v_{1,t}^\delta|-|r_{1,t}^\delta|)+K_1(|u_{0,t}|-|r_{1,t}^\delta|) +|\frac{1}{\eps}K_1 \tilde{q}_{1,t}| \\
	\p_t |r_{2,t}^\delta| \leq \frac{1}{\eps} (|v_{2,t}^\delta|-|r_{2,t}^\delta|)+K_2(|u_{0,t}|-|r_{2,t}^\delta|) + |\frac{1}{\eps} K_2 \tilde{q}_{2,t}| \\
	\qquad \qquad \quad +  |G'(q_2^\delta)\frac{1}{\eps}\tilde{q}_{2,t}| - G'(q_2^\delta)|r_{2,t}^\delta| \\
	\p_t |u_{0,t}| \leq K_1(|r_{1,t}^\delta|-|u_{0,t}|)+K_2 (|r_{2,t}^\delta|-|u_{0,t}|)+ 
	|\frac{1}{\eps}K_1 \tilde{q}_{1,t}|   \\
	\qquad \qquad \quad +|\frac{1}{\eps}K_2\tilde{q}_{2,t}| + |G'(q_2^\delta)\frac{1}{\eps}\tilde{q}_{2,t}| + G'(q_2^\delta) |r_{2,t}^\delta|.
	\end{cases}				
	\end{equation}
	Indeed, the right hand side of the $4th$ and $5th$ inequalities can be obtained as follows.
	On the one hand, we have
	\begin{align*}
	-G'(q_2^\delta)q_{2,t} \sgn(r_{2,t}^\delta) &
	=-G'(q_2^\delta)\left(r_{2,t}^\delta(t,x)-\frac{1}{\eps}\tilde{q}_{2,t}\left(\frac{t}{\eps},x\right)\right) \sgn(r_{2,t}^\delta)  \\
	&\leq -G'(q_2^\delta)\left|r_{2,t}^\delta\right|+ \frac{1}{\eps}\left|G'(q_2^\delta)\tilde{q}_{2,t}\right|.
	\end{align*}
	On the other hand
	\begin{align*}
	-G'(q_2^\delta)q_{2,t} \text{sgn}(u_{0,t}^\delta)& =-G'(q_2^\delta)\left(r_{2,t}^\delta(t,x)-\frac{1}{\eps}\tilde{q}_{2,t}\left(\frac{t}{\eps},x\right)\right) \sgn(u_{0,t})  \\
	& \leq G'(q_2^\delta)\left|r_{2,t}^\delta\right|+\frac{1}{\eps}\left|G'(q_2^\delta)\tilde{q}_{2,t}\right|,
	\end{align*}
	since $G$ is non-decreasing by assumption \eqref{nonlin}. 
	Summing all inequalities in \eqref{systemUQdt} and integrating with repsect to space on $(0,L)$, we obtain
	formally 
	\begin{equation}\label{dt}
	\frac{d}{dt} \tilde{\mathcal{H}}_{t}(t) + |v_{2,t}^\delta(t,0)| \leq F_1(t)+F_2(t)+F_3(t)+F_4(t),
	\end{equation}
	where 
	$$
	\begin{aligned}
	F_{1}(t)& :=|v_{2,t}^\delta(t,L)|-| v_{1,t}^\delta(t,L)|; \quad F_{2}(t):=| v_{1,t}^\delta(t,0)|, \\
	F_3(t)& :=\frac{1}{\eps}\int_{0}^{L}\left|\p_x \tilde{u}_{2,t}\left(\frac{t}{\eps},x\right) \right|  dx +\frac{1}{\eps}\int_{0}^{L}\left|\p_x \tilde{u}_{1,t}\left(\frac{t}{\eps},x\right) \right| dx,  \\
	F_{4}(t)&:=\frac{2K_1 }{\eps}\int_{0}^{L} \left|\tilde{q}_{1,t}\left(\frac{t}{\eps},x\right) \right|\,dx + \frac{2}{\eps}(\|G'\|_\infty+K_2)\int_{0}^{L}\left|\tilde{q}_{2,t}\left(\frac{t}{\eps},x\right) \right|\,dx.
	\end{aligned}
	$$
	
	Integrating \eqref{dt} in time, we get
	\begin{equation}\label{Htilde}
	\tilde{\mathcal{H}}_{t}(t) + \int_0^T|v_{2,t}^\delta(t,0)|\,dt \leq \int_0^T(F_1(t)+F_2(t)+F_3(t)+F_4(t))\,dt + \tilde{\mathcal{H}}_{t}(0).
	\end{equation}
	Let us consider each term of the right hand side of \eqref{Htilde} separately:
	\begin{itemize}
		\item $F_1$: On the right boundary $x=L$, one has
		\begin{align*}
		\int_{0}^{T}  (|& v_{2,t}^\delta(t,L)|- |v_{1,t}^\delta(t,L)|)\,dt
		\leq \ue \int_{0}^{T}\left| (\tilde{u}_{2,t}-\tilde{u}_{1,t})\left(\frac{t}{\eps},L\right)\right|\,dt  \\
		& \leq \int_{0}^{\frac{T}{\eps}}| (\tilde{u}_{2,t}-\tilde{u}_{1,t})(\tau,L)|\,d\tau  \leq \frac 12 \left\{  \left| u_1^{0,\delta}(0)-q_1^{0,\delta}(0)\right| + \left| u_2^{0,\delta}(0)-q_2^{\delta,0}(0)\right|\right\}\\
		&     \leq C \nrm{U^{0,\delta}}{\bW^{1,1}(0,L)}
		\end{align*}
		where we used trace operator's continuity 
		for $W^{1,1}(0,L)$ functions.
		\item $F_2$: On the other hand at $x=0$, the boundary condition can be estimated as
		\begin{align*}
		\int_{0}^{T}& |v_{1,t}^\delta(t,0)|\ dt
		\leq \int_{0}^{T} |(u_{b}^{\delta})'(s)|\,ds+ \int_{0}^{\frac{T}{\eps}} |(u_1^{0,\delta}(0)-q_1^{0,\delta}(0))e^{-2\tau}|\ d\tau \\
		& \leq C \left( \nrm{u_b^\delta}{W^{1,1}(0,T)} + \nrm{U^{0,\delta}}{\bW^{1,1}(0,L)}\right)
		%
		\end{align*}
		as above.
		\item $F_3$: With the change of variable $\tau=\frac{t}{\eps}$, we have, using again \eqref{eq.tilde},
		\begin{align*}
		\int_{0}^{T}\!\!\int_{0}^{L} \left|\frac{1}{\eps}\p_x  \tilde{u}_{i,t}\left(\frac{t}{\eps},x\right)\right|\,dxdt
		= \int_{0}^{\frac{T}{\eps}}\!\!\int_{0}^{L} |\p_x  \tilde{u}_{i,t}(\tau,x)|\,dxd\tau
		\leq C \nrm{U^{0,\delta}}{\bW^{1,1}(0,L)},
		\end{align*}
		which is uniformly bounded with respect to $\eps$.
		\item $F_4$: similarly, we have 
		\begin{align*}
		\int_0^T\int_0^L \left|\frac{1}{\eps} \tilde{q}_{i,t}\left(\frac{t}{\eps},x\right)\right|\,dxdt
		=   \int_{0}^{\frac{T}{\eps}}\int_{0}^{L} | \tilde{q}_{i,t}(\tau,x)|\ dxd\tau  
		\leq C \nrm{U^{0,\delta}}{\bW^{1,1}(0,L)}
		\end{align*}
		thanks to the fact that $\p_t \tilde{q}_{i}(\tau,x)=(q_i^{0}(x)-u_i^{0}(x))e^{-2\tau}$ with $i=1,2.$
	\end{itemize}
	It remains  to estimate  $\tilde{\mathcal{H}}_{t}(0)$ in \eqref{Htilde}. Indeed, 
	using \eqref{funsist} at $t=0$ in order to convert time derivatives into 
	expressions involving only the data and its space derivatives, one obtains  
	$$
	\tilde{\mathcal{H}}_{t}(0) = 
	\int_0^L \left(\sum_{i=1}^2  |v_{i,t}^\delta(0,x)|+ 
	|r_{i,t}^\delta(0,x)| + |u_{0,t}(0,x)| \right)\,dx \leq C \nrm{U^{0,\delta}}{\bW^{1,1}(0,L)}
	$$
	So for instance,
	for the first term of the sum, we use the first equation in \eqref{funsist} and we write
	$$
	\p_t v_1^\delta(0,x)=\frac{1}{\eps} (r_1^\delta(0,x)-v_1^\delta(0,x))+\p_x\tilde{u}_1(0,x)-\p_x v_1^\delta(0,x).
	$$
	Recalling that $r_1^\delta(0,x)=q_1^{0}(x)$ and $v_1^\delta(0,x)=q_1^{0}(x) $ as defined in \eqref{incond2}
	we get: $\p_t v_1^\delta(0,x)=\p_x (q_1^{0}(x)-u_1^{0}(x))-\p_x v_1^\delta(0,x)= -\p_x u_1^{0}(x),$ then
	$$
	\int_{0}^{L}|\p_t v_1^\delta(0,x)|\ dx \leq \int_{0}^{L}|\p_x u_1^{0,\delta}(x)|\ dx < C \nrm{U^{0,\delta}}{\bW^{1,1}(0,L)},
	$$
	The rest follows exactly the same way.
	We conclude from \eqref{Htilde} and the above calculations that
	$$
	\begin{aligned}
	\tilde{\mathcal{H}}_{t}(t)  + \int_0^T |v_{2,t}^\delta(t,0)|\,dt & \leq  \tilde{\mathcal{H}}_{t}(0)+\int_{0}^{t} (F_1+F_2+F_3+F_4)(s)\ ds \\
	& \leq  C\left(\nrm{U^{0,\delta}}{\bW^{1,1}(0,L)}+\nrm{u^\delta_b}{W^{1,1}(0,T)}\right).
	\end{aligned}
	$$
	Finally, in order to recover \eqref{estimU1t}, we add the first and third inequalities in \eqref{systemUQdt} 
	and integrate on $(0,T)\times(0,L)$, we get
	\begin{align*}
	&\int_0^L (|v_{1,t}^\delta(T,x)|+|r_{1,t}^\delta(T,x)|)\,dx + \int_0^T |v_{1,t}^\delta(t,L)|\,dt  \\
	&  \leq\ \int_0^T |v_{1,t}^\delta(t,0)|\,dt  + \int_0^L (|v_{1,t}^\delta(0,x)|+|r_{1,t}^\delta(0,x)|)\,dx  \\
	&\quad + \int_0^T\!\!\int_0^L \left(K_1|u^\delta_{0,t}(t,x)|
	+ \frac{K_1}{\eps}\left|\tilde{q}_{1,t}\left(\frac{t}{\e},x\right)\right| 
	+ \frac{1}{\eps}\left|\p_x\tilde{u}_{1,t}\left(\frac{t}{\e},x\right)\right| \right)\,dxdt.
	\end{align*}
	We have already proved that the second term of the right hand side is bounded.
	We have also proved above that $u^\delta_{0,t}$ is uniformly bounded in $L^\infty((0,T);L^1(0,L))$.
	From \eqref{incond2}, we have $v_{1,t}^\delta(t,0) = u_b'(t) + \frac{1}{\eps} \tilde{u}_{1,t}(\frac{t}{\eps},0)$. As above, we use the expressions of $\tilde{u}_1$ and $\tilde{q}_1$ and a change of variable to bound each term of the right hand side. 
\end{proof}

As a consequence, we deduce the following estimates on the time derivatives of the original unknowns
$(u_1^\delta,u_2^\delta,q_1^\delta,q_2^\delta,u_0^\delta)$~:
\begin{corollary}\label{lemmadt}
Let $T>0$, under the same assumptions,
there exists a constant $C_T>0$ depending only on the $W^{1,1}$ norm of the data but independent on $\e$,
such that~:
\begin{equation} \label{eq.estim.dt}
\begin{aligned}
& \int_0^T \int_{0}^{L} (|\p_{t}u_{1}^\delta|+|\p_{t}u_{2}^\delta|+|\p_{t}u_{0}^\delta|+|\p_{t}q_{1}^\delta|+|\p_{t}q_{2}^\delta|)(t,x)\,dx \,dt \leq C \nrm{U^{0,\delta}}{\bW^{1,1}(0,L)},  \\
& \int_{0}^{T} |\p_{t}u_{2}^\delta(t,0) |\ dt \leq C\nrm{U^{0,\delta}}{\bW^{1,1}(0,L)},  
\int_{0}^{T} |\p_{t} u_{1}^\delta(t,L)|\ dt \leq  C\nrm{U^{0,\delta}}{\bW^{1,1}(0,L)}. 
\end{aligned}  
\end{equation}
\end{corollary}		
\begin{proof}
	We recall the expressions
	$$v^\delta_1=u^\delta_1+\tilde{u}_1, \quad  v^\delta_2=u^\delta_2+\tilde{u}_2,  \quad r^\delta_1=q^\delta_1+\tilde{q}_1, \quad r^\delta_2=q^\delta_2+\tilde{q}_2. $$
	By the triangle inequality, we have for $i\in \{1,2\}$,
	\begin{align*}
	\|\p_t u^\delta_i\|_{L^1([0,T]\times[0,L]) }\leq \|\p_t v^\delta_i\|_{L^1([0,T]\times[0,L]) }+\frac 1\eps \|\p_t \tilde{u}_i(t/\eps,x) \|_{L^1([0,T]\times[0,L]) },  \\
	\|\p_t q^\delta_i\|_{L^1([0,T]\times[0,L]) }\leq \|\p_t r^\delta_i\|_{L^1([0,T]\times[0,L]) }+\frac 1\eps \|\p_t \tilde{q}_i(t/\eps,x) \|_{L^1([0,T]\times[0,L]) }.
	\end{align*}
	The first terms of the latter right hand side are bounded from Proposition \ref{propdt}. For the second terms, we have, as above,
	\begin{align*}
	\int_{0}^{T}\int_{0}^{L} \frac 1\eps \left|\p_t \tilde{q}_{i}\left(\frac{t}{\eps},x\right)\right|\ dxdt = \frac{1}{\eps} 
	\int_{0}^{T}\int_{0}^{L} \left|(q_i^{0}(x)-u_i^{0}(x))e^{-2\frac{t}{\eps}}\right|\ dxdt <  C \nrm{U^{0,\delta}}{\bW^{1,1}(0,L)},  \\
	\int_{0}^{T}\int_{0}^{L} \frac{1}{\eps}\left|\p_t \tilde{u}_{i}\left(\frac{t}{\eps},x\right)\right|\ dxdt = \frac{1}{\eps} 
	\int_{0}^{T}\int_{0}^{L}\left|(u_i^{0}(x)-q_i^{0}(x))e^{-2\frac{t}{\eps}}\right|\ dxdt < C' \nrm{U^{0,\delta}}{\bW^{1,1}(0,L)}.
	\end{align*}
	
	Furthermore, from \eqref{Htilde}, we get
	$$
	\int_{0}^{T} \left|v^\delta_{2,t}(t,0)\right|\ dt \leq C_T \nrm{U^{0,\delta}}{\bW^{1,1}(0,L)}.
	$$
	By the triangle inequality, it implies the second estimate in \eqref{eq.estim.dt}
	
	To recover the third claim in \eqref{eq.estim.dt},
	we notice that by definition of $v_1^\delta$ and a triangle inequality, we have
	$$
	|u^\delta_{1,t}(t,L)|\leq |v^\delta_{1,t}(t,L)| 
	+ \frac 1\eps \left|\tilde{u}_{1,t}\left(\frac{t}{\eps},L\right)\right| \leq |v^\delta_{1,t}(t,L)| + \frac 1\eps \|q^{0,\delta}_ 1- u^{0,\delta}_1\|_{W^{1,1}(0,L)} e^{-2\frac{t}{\eps}}.
	$$
	where again we use the continuity of the trace operator on $W^{1,1}(0,L)$ functions in order to recover the 
	dependence between the values at $x=L$ and the $W^{1,1}(0,L)$-norm of the initial data. 
	Integrating in time and using \eqref{estimU1t} allows  to conclude.
\end{proof}

\subsubsection{Uniform bounds on the space derivatives}\label{sec.bound.space}

\begin{lem}\label{lem:estimdx}
	Let $T>0$.
	If the data is regular in the sense of Definition \ref{def.reg.data},
	then, the space derivatives of functions $u^\delta_1$, $u^\delta_2$ satisfy the following estimates :
	$$
	\int_0^T\int_{0}^{L} \sum_{i=0}^2 |\p_{x}u^\delta_{i}(t,x)|+\sum_{i=1}^2 |\p_{x}q^\delta_{i}(t,x)| \,dxdt \leq C_T \nrm{U^{0,\delta}}{W^{1,1}((0,L))},
	$$
	for some non-negative constant $C_T$ uniformly bounded with respect to $\eps$.
\end{lem}				

\begin{proof}
	Adding equation \eqref{fivea} with \eqref{fivec} and also \eqref{fiveb} with \eqref{fived} we get
	\begin{align*}
	\alpha \p_x u^\delta_1 & = K_1(u^\delta_0-q^\delta_1) - \p_t u^\delta_1 -\p_t q^\delta_1,  \\
	-\alpha \p_x u^\delta_2 & = K_2(u^\delta_0-q^\delta_2) - \p_t u^\delta_2  -\p_t q^\delta_2 - G(q^\delta_2).
	\end{align*}
	Using Corollary \ref{lemmadt} and \eqref{nonlin}, the right hand sides are uniformly bounded in $L^1((0,T)\times(0,L))$.
	Deriving the ODE part of \eqref{five} with respect to the space variable and using the latter estimates
	provides the results for $u_0^\delta$, $q_1^\delta$ and $q_2^\delta$.
\end{proof}

\subsection{Extension to $\BV$ data}
We show here how to use Corollary \ref{lemmadt} and Lemma \ref{lem:estimdx} in order to obtain $\BV$ compactness.
\begin{theo}
	Under hypotheses \eqref{ass.bvcond}-\eqref{ass.nonlin}, 
	there exists a uniform bound such that the $\eps$-dependent solutions of system \eqref{five}  satisfy
	$$
	\sum_{i=0}^2 \nrm{u_i}{BV((0,T)\times(0,L))} + \sum_{i=1}^2 \nrm{q_i}{BV((0,T)\times(0,L))} 
	\leq C \left(\nrm{U^0}{\bBV(0,L)} +    \nrm{u_b}{BV(0,T)}   \right)
	$$
	where the generic constant $C$ is independent on $\e$.
\end{theo}

\begin{proof}
	Setting $U^\delta(t,x) := (u_1^\delta(t,x) ,u_2^\delta(t,x) ,q_1^\delta(t,x) ,q_2^\delta(t,x) ,u_0^\delta(t,x) )$, one has from the previous estimates :
	$$
	\nrm{U^\delta}{\bW_{t,x}^{1,1}((0,T)\times(0,L))} \leq 
	C \left( \nrm{U^{0,\delta}}{\bW_{x}^{1,1}(0,L)}  + \nrm{u_b^\delta}{W^{1,1}(0,T)} \right)
	$$
	Now using Theorem 5.3.3 \cite{ziemer} one estimates the rhs with respect to the $\BV$ norm of the data :
	$$
	\begin{aligned}
	\nrm{U^{0,\delta}}{\bW_{x}^{1,1}(0,L)}  \leq &
	\nrm{(1-\chi_{\delta}-\chiu{\delta}(L-\cdot) )u_1^0 + \chi_{\delta} c_1}{BV(0,L)} +\\
	& +\nrm{(1-\chiu{\delta}(L-\cdot) )u_2^0 + \chi_{\delta}(L-\cdot) c_2}{BV(0,L)} 
	+ \nrm{(q_1^0,q_2^0,u_0^0)}{BV(0,L)^4} \\
	\leq & \nrm{U^0}{\bBV(0,L)} + \nrm{ \chi_{\delta} c_1}{BV(0,L)}+ \nrm{ \chi_{\delta}(L-\cdot) c_2}{BV(0,L)} ,\\
	\nrm{u_b^\delta}{W^{1,1}(0,T)}  \leq&  \nrm{u_b}{\BV(0,T)}	 + c_1 \nrm{\chi_\delta}{\BV(0,T)}
	\end{aligned}
	$$
	A simple computation shows that $
	\nrm{\chi_{\delta}}{BV(0,\max(T,L))} < C$ uniformly with respect to $\delta$. Then choosing $c_i= \nrm{U^0}{\bBV(0,L)} + \nrm{u_b}{\BV(0,T)}$ for $i \in \{1,2\}$ shows that
	$$
	\nrm{U^{0,\delta}}{\bW_{x}^{1,1}(0,L)} \leq C \left(\nrm{U^0}{\bBV(0,L)} + \nrm{u_b}{\BV(0,T)} \right) 
	$$
	We are in the hypotheses of Theorem 5.2.1. p. 222 of \cite{ziemer} : by  $L^1$-continuity, shown in Theorem \ref{th:exist},
	$U^\delta$ tends to $U := (u_1,u_2,q_1,q_2,u_0)$  in $L^1((0,T)\times(0,L))$ strongly, when $\delta$ vanishes. Then for any open
	set $V \subset (0,T)\times(0,L)$ one has 
	$$
	\nrm{U}{\bBV_{t,x}(V)} \leq \liminf_{\delta \to 0} \nrm{U^\delta}{\bBV_{t,x}((0,T)\times(0,L))}
	= \liminf_{\delta \to 0} \nrm{U^\delta}{\bW^{1,1}_{t,x}((0,T)\times(0,L))}
	$$
	and since the $\bBV$ bound of the sequence $(U^\delta)_{\delta}$ is uniformly bounded with respect to $\delta$, 
	the result extends by Remark 5.2.2. p. 223 \cite{ziemer} to the whole set $(0,T)\times(0,L)$.
\end{proof}

\section{Proof of the convergence result (Theorem \ref{th:conv})}\label{sec:convergence}


	It is divided into two steps. 
	\begin{enumerate}
		\item {\bf Convergence :} 
		From Lemma \ref{lem:infty}, Lemma \ref{prop} and Corollary \eqref{lemmadt}, the sequences $(u_1^\eps)_\eps$ and $(u_2^\eps)_\eps$ are uniformly bounded in $L^\infty\cap BV((0,T)\times (0,L))$.
		Thanks to the Helly's theorem (\cite{dafermos, lefloch}), we deduce that, up to extraction of a subsequence, 
		\begin{align*}
		& u_{1}^{\eps} \xrightarrow[\eps \to 0]{} u_1 \text{ strongly in } L^1([0,T]\times[0,L]), \\
		& u_{2}^{\eps} \xrightarrow[\eps \to 0]{} u_2 \text{ strongly in } L^1([0,T]\times[0,L]),
		\end{align*}
		with limit function $u_1,u_2 \in L^\infty\cap BV((0,T)\times(0,L))$.
		
		By equations \eqref{fivea}, one shows by testing with the appropriate $C^1$ compactly supported 
		functions in $(0,T)\times(0,L)$ and using the definition of the $\BV$ norm (cf for instance,\cite{ziemer} p. 220--221):
		$$
		\|q_1^{\eps}- u_1^{\eps}\|_{L^1((0,T)\times(0,L))}\leq C\eps \nrm{u_1^\eps}{BV((0,T)\times(0,L))}
		$$ 
		which tends to zero as $\eps$ goes to $0$ thanks to the bounds in Corollary \ref{lemmadt} and Lemma \ref{lem:estimdx}. 
		Therefore, $q_{1}^{\eps} \xrightarrow[\eps \to 0]{} u_1$ strongly in $L^{1}((0,T)\times(0,L))$. 
		By the same argument, we show that $q_{2}^{\eps} \xrightarrow[\eps \to 0]{} u_2$ strongly in $L^{1}((0,T)\times(0,L))$.
		Moreover, since $G$ is Lipschitz-continuous, we have, when $\eps$ goes to zero,
		$$
		\|G(q_2^{\eps})-G(u_2)\|_{L^{1}((0,T)\times(0,L))} \longrightarrow 0.
		$$ 
		
		For the convergence of $u_0^{\eps}$, let us first denote $u_0$ a solution to the equation
		$$
		\p_{t}u_0=K_1(u_1-u_0) + K_2(u_2-u_0) + G(u_2).
		$$
		Then, taking the last equation of system \eqref{five},
		subtracting by this latter equation and multiplying by $\sgn(u_0^{\eps}-u_0)$, we get
		in a weak sense that 
		\begin{align*}
		\p_{t}|u_0^{\eps}-u_0|
		& \leq K_1|q_1^{\eps}-u_1| + K_2|q_2^{\eps}-u_2| +(K_1+K_2)|u_0-u_0^{\eps}|
		+ |G(q_2^{\eps})-G(u_2)|  \\
		& \leq K_1|q_1^{\eps}-u_1| + K_2|q_2^{\eps}-u_2| +(K_1+K_2)|u_0-u_0^{\eps}|
		+ \|G'\|_\infty |q_2^{\eps}- u_2|.
		\end{align*}
		Using a Gr\"onwall Lemma, we get, after an integration on $[0,L]$,
		\begin{align*}
		\int_{0}^{L}|u_0^{\eps}-u_0|(t,x)\,dx \leq
		&\ \int_{0}^{L} e^{(K_1+K_2)t}|u_0-u_0^{\eps}|(0,x)\,dx \\
		& + K_{1}\int_{0}^{L}\!\!\int_{0}^{T}e^{(K_1+K_2)(t-s)}|q_1^{\eps}-u_1|(s,x)\,dsdx  \\
		& + (\|G'\|_\infty+ K_2)\int_{0}^{L}\!\!\int_{0}^{T}e^{(K_1+K_2)(t-s)}|q_2^{\eps}-u_2|(s,x)\,dsdx.
		\end{align*}
		Thus, one shall conclude that
		$$
		u_{0}^{\eps} \xrightarrow[\eps \to 0]{} u_0 \text{ strongly  in } L^1((0,T)\times(0,L)).
		$$
		
		\item{\bf The limit system :} \\
		We pass to the limit in \eqref{eq.weak.form.eps},  the weak formulation of system \eqref{five}.
		Suppose that $\bphi \in \cS_5$. 
		Taking $\phi_1=\phi_3$ and $\phi_{2} = \phi_{4}$ in \eqref{eq.weak.form.eps} 
		we may pass to the limit $\eps\to 0$ and we obtain 
		$$
		\begin{aligned}
		&\int_{0}^{T} \int_{0}^{L} u_1(2 \p_t \phi_1+\alpha \p_x \phi_1) dxdt +\int_{0}^{T} u_b(t)\phi_1(t,0)\ dt
		+ \int_{0}^{L} (u_1^0(x)+q_1^0(x))\phi_1(0,x)\ dx \\
		&+  \int_{0}^{T} \int_{0}^{L} u_2( 2 \p_t \phi_2-\alpha \p_x \phi_2) 
		+ \int_{0}^{L} (u_2^0(x)+ q_2^0(x))\phi_2(0,x)\ dx \\
		& + \int_{0}^{T} \int_{0}^{L} u_0 \p_t \phi_3  
		+K_1(u_1-u_0)  (\phi_3 -\phi_1 )
		+K_2 (u_2-u_0)  (\phi_3 -\phi_2 ) 
		+ G(u_2) (\phi_3- \phi_2)\ dxdt \\
		&  +\int_{0}^{L} u_0(0,x) \phi_3(0,x)\ dx=0. \\ 
		\end{aligned}
		$$
		which is exactly \eqref{eq.weak.form.zero} with initial data coming from system \eqref{five}.
		Finally, since the solution of the limit system is unique, we deduce that the whole sequence converges.
		This concludes the proof of Theorem \ref{th:conv}.
	\end{enumerate}

\section{Conclusion} \label{sec.cl}

In this study we presented a model describing the transport of ionic concentrations, in particular sodium, for a simplified version of the loop of Henle in a kidney nephron. After introducing the system, we dealt with the rigorous passage to the limit in semi-linear hyperbolic
$5\times 5$ system, accounting for the presence of epithelium layers, towards a $3\times3$ system \eqref{systlim1}-\eqref{systlim3}. 

Physically, studying the asymptotic with respect to parameter $\eps$ (accounting for permeability) means to consider  very large permeabilities. Roughly speaking, taking into account the limit when $\eps$ goes to 0, means 'removing' the epithelial layers and assuming that the tubule is  directly in contact with the surrounding interstitium.
This work ensures consistency between the reduced model and the 'epithelial' model and also rigorously explains and makes explicit the link between two possible descriptions of the same physical phenomenon, but with two different levels of complexity.

The reduced system has already given a proper representation of the counter-current mechanism, but it is not sufficient to give other suggestions about the description of the entire phenomenon and, for example, about sodium fluxes in clinical cases. 
As already discussed in \cite{MAMV}, despite the addition of the epithelial layer, the model remains far from how the nephron and kidneys actually work.

In order to look after a more appropriate analysis regarding physiological conditions, the first step would be to take into account
water flow and the fluid reabsorption in the descending tubule. Then, the second one would be to consider the electrical forces that apply to ions such as sodium and potassium, and that modulate the flows which depend not only on concentration gradients but also on electrical potential, \cite{laytonedwards, magali}. \\
The system would give a contribution in the
field of physiological renal transport model and it could be a good starting point to elucidate and to better understand some mechanisms underlying concentrating mechanism and the transport of ions in the kidney.

\bibliographystyle{abbrv}
\bibliography{biblioMMV}

\begin{thebibliography}{10}

\bibitem{dafermos}
C.~M. {Dafermos}.
\newblock {\em {Hyperbolic conservation laws in continuum physics. 4th
  edition.}}, volume 325.
\newblock Berlin: Springer, 4th edition edition, 2016.

\bibitem{GioYaYo}
V.~{Giovangigli}, Z.-B. {Yang}, and W.-A. {Yong}.
\newblock {Relaxation limit and initial-layers for a class of
  hyperbolic-parabolic systems.}
\newblock {\em {SIAM J. Math. Anal.}}, 50(4):4655--4697, 2018.

\bibitem{giusti}
E.~{Giusti}.
\newblock {\em {Minimal surfaces and functions of bounded variation.}},
  volume~80.
\newblock Birkh\"auser/Springer, Basel, 1984.

\bibitem{HeiPaRe}
M.~{Heida}, R.~I.~A. {Patterson}, and D.~R.~M. {Renger}.
\newblock {Topologies and measures on the space of functions of bounded
  variation taking values in a Banach or metric space.}
\newblock {\em {J. Evol. Equ.}}, 19(1):111--152, 2019.

\bibitem{James}
F.~{James}.
\newblock {Convergence results for some conservation laws with a reflux
  boundary condition and a relaxation term arising in chemical engineering.}
\newblock {\em {SIAM J. Math. Anal.}}, 29(5):1200--1223, 1998.

\bibitem{JinXin}
S.~{Jin} and Z.~{Xin}.
\newblock {The relaxation schemes for systems of conservation laws in arbitrary
  space dimensions.}
\newblock {\em {Commun. Pure Appl. Math.}}, 48(3):235--276, 1995.

\bibitem{anitalayton}
A.~T. Layton.
\newblock {{M}athematical modeling of kidney transport}.
\newblock {\em Wiley Interdiscip Rev Syst Biol Med}, 5(5):557--573, 2013.

\bibitem{laytonedwards}
A.~T. Layton and A.~Edwards.
\newblock {\em Mathematical Modeling in Renal Physiology}.
\newblock Springer, 2014.

\bibitem{lefloch}
P.~G. {LeFloch}.
\newblock {\em {Hyperbolic systems of conservation laws. The theory of
  classical and nonclassical shock waves.}}
\newblock Basel: Birkh\"auser, 2002.

\bibitem{MAMV}
M.~Marulli, A.~Edwards, V.~Milišić, and N.~Vauchelet.
\newblock On the role of the epithelium in a model of sodium exchange in renal
  tubules.
\newblock {\em Mathematical Biosciences}, 321:108308, 2020.

\bibitem{MiOel.1}
V.~{Mili\v{s}i\'c} and D.~{Oelz}.
\newblock {On the asymptotic regime of a model for friction mediated by
  transient elastic linkages.}
\newblock {\em {J. Math. Pures Appl. (9)}}, 96(5):484--501, 2011.

\bibitem{nataliniterracina}
R.~{Natalini} and A.~{Terracina}.
\newblock {Convergence of a relaxation approximation to a boundary value
  problem for conservation laws.}
\newblock {\em {Commun. Partial Differ. Equations}}, 26(7-8):1235--1252, 2001.

\bibitem{perth}
B.~{Perthame}.
\newblock {\em {Transport equations in biology.}}
\newblock Basel: Birkh\"auser, 2007.

\bibitem{PST}
B.~{Perthame}, N.~{Seguin}, and M.~{Tournus}.
\newblock {A simple derivation of BV bounds for inhomogeneous relaxation
  systems.}
\newblock {\em {Commun. Math. Sci.}}, 13(2):577--586, 2015.

\bibitem{magali}
M.~Tournus.
\newblock {\em Mod{\`e}les d'{\'e}changes ioniques dans le rein: th{\'e}orie,
  analyse asymptotique et applications num{\'e}riques}.
\newblock PhD thesis, Université Pierre et Marie Curie, France, 2013.

\bibitem{tesp}
M.~{Tournus}, A.~{Edwards}, N.~{Seguin}, and B.~{Perthame}.
\newblock {Analysis of a simplified model of the urine concentration
  mechanism.}
\newblock {\em {Netw. Heterog. Media}}, 7(4):989--1018, 2012.

\bibitem{magali2}
M.~Tournus, N.~Seguin, B.~Perthame, S.~R. Thomas, and A.~Edwards.
\newblock {{A} model of calcium transport along the rat nephron}.
\newblock {\em Am. J. Physiol. Renal Physiol.}, 305(7):F979--994, Oct 2013.

\bibitem{ziemer}
W.~P. {Ziemer}.
\newblock {\em {Weakly differentiable functions. Sobolev spaces and functions
  of bounded variation.}}, volume 120.
\newblock Berlin etc.: Springer-Verlag, 1989.

\end{thebibliography}

\end{document}